\documentclass [11pt]{amsart}
\usepackage {amsmath, amssymb, amscd, mathrsfs, url,pinlabel,hyperref, verbatim}
\usepackage[text={6.5in,9in},centering,letterpaper,dvips]{geometry}
\usepackage{color,multirow,dcpic,latexsym,pictexwd,graphicx,epstopdf}

\newtheorem {theorem}{Theorem}
\newtheorem {lemma}[theorem]{Lemma}
\newtheorem {proposition}[theorem]{Proposition}
\newtheorem {corollary}[theorem]{Corollary}
\newtheorem {conjecture}[theorem]{Conjecture}
\newtheorem {definition}[theorem]{Definition}

\newtheorem {question}[theorem]{Question}
\newtheorem {problem}[theorem]{Problem}
\theoremstyle{remark}
\newtheorem {claim}[theorem]{Claim}
\newtheorem {remark}[theorem]{Remark}

\numberwithin{equation}{section}
\numberwithin{theorem}{section}

\newcommand{\HFred}{HF_{\text{red}}}
\newcommand{\HFhat}{\widehat{HF}}
\DeclareMathOperator{\rank}{rank}
\newcommand{\comments}[1]{}
\newcommand{\spinc}{\mathfrak{s}}

\DeclareMathOperator{\coker}{coker}
\DeclareMathOperator{\lcm}{lcm}

\title{Rank inequalities for the Heegaard Floer homology of Seifert homology spheres}
\author{\c{C}a\u gr\i\; Karakurt and Tye\; Lidman}
\date{}

\begin{document}
\maketitle

\begin{abstract}
We establish three rank inequalities for the reduced flavor of Heegaard Floer homology of Seifert fibered integral homology spheres.  Combining these inequalities with the known classifications of non-zero degree maps between Seifert fibered spaces, we prove that a map $f:Y' \to Y$ between Seifert homology spheres yields the inequality $ |\deg f| \rank \HFred(Y)  \leq \rank \HFred(Y')$.  These inequalities are also applied in conjunction with an algorithm of N\'emethi to give a method to solve the botany problem for the Heegaard Floer homology of these manifolds.
\end{abstract}

\section{Introduction}
In the past several years, a great deal of progress has been made in the combinatorial description of the Heegaard Floer invariants of various objects in low-dimensional topology.  Many such algorithms come from the use of special Heegaard diagrams representing the objects in question, beginning with the advent of nice diagrams for closed $3$-manifolds in \cite{SW}.  This idea was used for knots and links in $S^3$ \cite{MOS, MOST}, which was then extended in \cite{MOT} to give a completely combinatorial description of the Heegaard Floer homology of closed three-manifolds and the Ozsv\'ath-Szab\'o invariants of closed, smooth $4$-manifolds with $b^+_2\geq 2$. 

\vspace{0.2cm}
   
In the case of Seifert fibered spaces, there is another combinatorial method to compute Heegaard Floer homology more effectively. This program was initiated by Ozsv\'ath and Szab\'o in their beautiful  paper \cite{OS1}, where they showed, by incorporating adjunction relations \cite{OS3}, that the Heegaard Floer homology of a $3$-manifold which bounds a certain type of plumbing in fact depends only on the intersection form of the plumbed $4$-manifold. They also found an algorithm which calculates the subgroup which is the kernel of the $U$ map on the plus version of  Heegaard Floer homology. Ozsv\'ath and Szab\'o's  result was extended by N\'emethi in \cite{N} which resulted in a fast algorithm calculating the full Heegaard Floer homology of any $3$-manifold which bounds a so--called almost rational plumbing. In the special case where the $3$-manifold is a Seifert  fibered integer homology sphere (or for short Seifert homology sphere), Can  and  the first author reformulated N\'emethi's algorithm in terms of a semigroup which is generated by a combination of Seifert invariants \cite{CK}. Despite this progress, no  closed formula is known  for the Heegaard Floer homology of Seifert homology spheres in terms of their Seifert invariants.  

\vspace{0.2cm}

The main purpose of the present article is to develop some combinatorial tools to compare the  Heegaard Floer homologies of two given Seifert  homology spheres without actually calculating them explicitly. We shall prove three  rank inequalities using these tools. Two of these inequalities arise from some geometric instances, namely the existence of certain kinds of maps between the spaces, but we never directly use these maps in our argument. Presumably one can prove  more general rank inequalities  (or possibly slightly different versions)  by incorporating these maps with a compatible version of Floer homology.  This strategy has been carried out for certain types of covering maps \cite{LM, LT}.  

\newpage

We start by stating the claimed inequalities. Henceforth $\Sigma(p_1,\ldots,p_l)$ denotes the Seifert homology sphere corresponding to a given $l$-tuple of pairwise relatively prime positive integers $(p_1,\ldots,p_l)$ with $p_i\geq 2$ for all $i=1,\dots,l$. Recall that the reduced version of Heegaard Floer homology  $\HFred(\Sigma(p_1,\ldots,p_l))$ is a finitely generated abelian group \cite{OS4}, and hence it has a well-defined rank.  We also recall that $S^3$ is the only Seifert homology sphere with $l \leq 2$ and that $\HFred(S^3) = 0$.

\vspace{0.2cm}

\subsection{Rank inequalities}

The first rank inequality concerns a particular type of branched cover between Seifert homology spheres. Let $Y = \Sigma(p_1,\ldots,p_l)$.  Fix $n \in \mathbb{Z}$ relatively prime to $p_1,\ldots,p_{l-1}$ and let $Y' = \Sigma(p_1,\ldots,p_{l-1},n p_l)$. The manifold $Y'$ is the $n$-fold cyclic branched cover of $Y$ branched along the singular fiber of order $p_l$. 

\vspace{0.2cm}

\begin{theorem}\label{thm:fiberbranched} (Rank inequality for branched covers along singular fibers)
 We have 
\begin{equation}\label{eq:fiberbranched}
n\left [ \rank \HFred(\Sigma(p_1,\ldots,p_l)) \right ] \leq \rank \HFred(\Sigma(p_1,\ldots,p_{l-1},n p_l)).  
\end{equation}

\end{theorem}

\vspace{0.2cm}

There also exists a   certain degree one map $f:\Sigma(p_1,\ldots,p_l) \to \Sigma(p_1,\ldots,p_k,p_{k+1} \cdots p_l)$ called  a vertical pinch. See Section \ref{sec:nonzerodegree} for a geometric description. Our second inequality shows that Heegaard Floer homology is sensitive to the existence of vertical pinches.  

\vspace{0.2cm}

\begin{theorem}\label{thm:pinch} (Rank inequality for vertical pinches)
Let $l \geq 4$ and fix $2 \leq k \leq {l-2}$.  Then
\begin{equation}\label{eq:pinch}
 \rank \HFred(\Sigma(p_1,\ldots,p_k,p_{k+1} \cdots p_l) ) \leq  \rank \HFred(\Sigma(p_1,\ldots,p_l)).
\end{equation} 
\end{theorem}

\vspace{0.2cm}

\begin{remark}
Recall that the rank of the Instanton Floer homology of a Seifert homology sphere equals  its Casson invariant \cite{FS}. The splice-additivity of the Casson invariant, \cite{ FMK,NW},  implies the analogue of Theorem~\ref{thm:pinch} for Instanton Floer homology. 
\end{remark}

\vspace{0.2cm}

  Consider the following partial order on the set of $l$-tuples of integers.  We will write $ (p_1,\ldots,p_l) \leq  (q_1,\ldots,q_l)$, if there exists a permutation $\sigma$ of the set $\{1,\dots,l\}$ such that $p_i \leq q_{\sigma(i)}$ for all $i=1,\dots,l$.  This relation  naturally induces a partial order on the set of Seifert homology spheres with $l$ singular fibers. The following result states that the rank of the reduced Heegaard Floer homology is monotone under this partial order.

\vspace{0.2cm}

\begin{theorem}\label{thm:monotonicity} (Partial order rank inequality)
Let $Y = \Sigma(p_1,\ldots,p_l)$ and $Y' = \Sigma(q_1,\ldots,q_l)$ be Seifert homology spheres such that $(p_1,\ldots,p_l) \leq  (q_1,\ldots,q_l)$.  Then, we have 
\begin{equation}\label{eq:monotonicity}
\rank \HFred(Y)  \leq \rank \HFred(Y').
\end{equation}
\end{theorem}

\vspace{0.2cm}

\subsection{The botany problem for Seifert homology spheres} Rank inequalities are  useful when one studies global problems about Heegaard Floer homology. Recall that the plus flavor of Heegaard Floer homology of any integral homology sphere is a  $\mathbb{Z}$-graded $\mathbb{Z}[U]$-module.

\begin{question}  Let $M$ be a $\mathbb{Z}$-graded $\mathbb{Z}[U]$-module. Is there a Seifert homology sphere realizing $M$ as its Heegaard Floer homology? If there is at least one, what are all the Seifert homology spheres whose Heegaard Floer homology is isomorphic to $M$?
\end{question}

Of course one can ask many different versions of this question. For example instead of Seifert homology spheres, one can take any family of $3$-manifolds. We focused on Seifert homology spheres because  their Heegaard Floer homology can be calculated easily using an algorithm \cite{N}. Nevertheless the algorithm itself is not sufficient for problems involving infinite families. As an application of our rank inequalities we prove the following result.

\newpage

\begin{theorem}\label{thm:botany}
Suppose $\rank \HFred (\Sigma (p_1,\ldots,p_l))=n \geq 1$, then
\begin{enumerate}
\item $6 \leq l!< \max \{2n ,7 \}$,\label{i:bot1}
\item $\max\{p_1,\ldots,p_l \} < 6n+7$.\label{i:bot2}
\end{enumerate}
\end{theorem}

Given $n\geq 1$, there are only finitely many tuples of pairwise relatively prime integers $(p_1,\ldots,p_l)$ satisfying the conditions in the above theorem. Hence one can solve the botany problem by calculating the Heegaard Floer homology of a finite set of Seifert homology spheres. To illustrate this, we list  all Seifert homology spheres with rank at most $12$ in Table~\ref{tab:rank}.

\begin{remark}
Previously it was shown in \cite{CK} that there can be only finitely many Seifert homology spheres with fixed Heegaard Floer homology (as a graded $\mathbb{Z}[U]$-module), without explicitly giving bounds on the number of singular fibers and their multiplicities. 
\end{remark}

\subsection{Non-zero degree maps between Seifert homology spheres}
\begin{question}
Given $3$-manifolds $Y$ and $Y'$, does there exist a map $f:Y'\to Y$ with $\mathrm{deg}(f)\neq 0$?
\end{question}

We will not attempt to answer this question here, but our work provides  evidence that Heegaard Floer homology could be useful in this direction. Indeed, in order to show that such a map does not exist, one should look for some obstructions.   Various topological quantities obstruct the  existence of a non-zero degree map. The rank of the first singular homology (i.e. the first betti number) is such a quantity, as one can easily check with elementary algebraic topology. Then a natural question is whether the rank of the Heegaard Floer homology  gives a similar obstruction. One of the goals of this paper is to study the behavior under non-zero degree maps of the Heegaard Floer homology of Seifert homology spheres.  It turns out that  such maps are  well-understood \cite{Hu, Ro2, Ro1}. By combining these results with our rank inequalities we prove the following theorem.

\begin{theorem}\label{thm:nonzeroexist}
Let $f:Y' \to Y$ be a map between Seifert homology spheres. Then 
\begin{equation}\label{eqn:rankinequality}
\left | \mathrm{deg}(f) \right |  \rank \HFred(Y) \leq  \rank \HFred(Y').
\end{equation}
\end{theorem}

\noindent Note that Theorem~\ref{thm:nonzeroexist} is trivial if $\mathrm{deg}(f) = 0$.

\vspace{0.2cm}

The organization is as follows: In Section~\ref{s:Calc}, we review N\'emethi's method for calculating the Heegaard Floer homology of Seifert homology spheres. In Section~\ref{s:absdel}, we build the theory of abstract delta sequences and their morphisms to develop the combinatorial machinery to prove the three rank inequalities. Inequalities (\ref{eq:fiberbranched}), (\ref{eq:monotonicity}), and (\ref{eq:pinch}) are proven in Section~\ref{sec:bra}, Section~\ref{sec:par}, and Section~\ref{sec:pin} respectively.  We discuss the applications to the botany problem and  nonzero degree maps in Section~\ref{s:botany} and Section~\ref{sec:nonzerodegree} respectively. In the last section we address some further directions and open problems.

\section*{Acknowledgments}
We would like to thank Ian Agol for responding to our question on MathOverflow, leading to a proof of Proposition \ref{prop:classificationbranched}.  In the course of this work, the first author  was supported by a Simons Fellowship and the National Science Foundation FRG Grant DMS-1065178.  The second author was partially supported by the National Science Foundation RTG Grant DMS-0636643.

\section{Calculating the Heegaard Floer homology of Seifert homology spheres}\label{s:Calc}

Graded roots are certain infinite trees that naturally encode Heegaard Floer homology \cite{N}. These objects can be described by sequences as follows.  Let $\tau$ be a given sequence of integers which is either finite or  non-decreasing after a finite index $N$. For every $n\in \mathbb{N}$, let $R_n$ be the infinite graph with vertex set $\mathbb{Z}\cap [\tau (n), \infty)$ and the edge set $\{[k,k+1]: k \in \mathbb{Z}\cap [\tau (n), \infty) \}$. We identify all common vertices and edges in $R_n$ and $R_{n+1}$ for each $n\in \mathbb{N}$ to get an infinite tree $\Gamma_\tau$. To each vertex $v$ of $\Gamma_\tau$, we can assign a grading $\chi_\tau(v)$ which is the unique integer corresponding to $v$ in any $R_n$ to which $v$ belongs. The pair $(\Gamma_\tau, \chi_\tau)$ is called a \emph{graded root}. Most of the time, we drop the grading function $\chi_\tau$ from our notation for brevity.   Clearly many different sequences can give the same graded graded root. For example $\Gamma_\tau$ does not depend on the values $\tau(n)$ for $n>N$. In fact, $\Gamma_\tau$ is completely determined by the subsequence of local maximum and local minimum  values of $\tau$. See Figure~\ref{fig:gradedroot} for an example of a graded root given by the sequence $\tau=(-2,-1,-2,0,-2,\dots)$ where $\tau$ is increasing after the first five terms.

\begin{figure}[h]
	\includegraphics[width=0.40\textwidth]{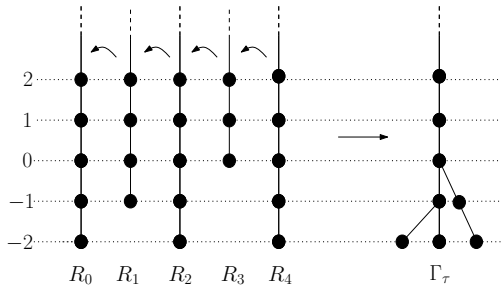}
	\caption{Graded root for $\tau=(-2,-1,-2,0,-2,\nearrow)$.}
	\label{fig:gradedroot}
\end{figure}

To any graded root $\Gamma_\tau$, we associate a $\mathbb{Z}$-graded $\mathbb{Z}[U]$-module as follows: Let $\mathbb{H} (\Gamma_\tau)$ be the free $\mathbb{Z}$-module on the vertex set of $\Gamma _\tau$.  We require that the degree of the generator corresponding to each vertex $v$  has degree $2 \chi_\tau (v)$. We define a degree $-2$ endomorphism $U$ of $\mathbb{H} (\Gamma_\tau)$ by sending each vertex $v$ to the sum of the vertices $w$ where  $w$ is connected to $v$ by an edge and $\chi_\tau(w)<\chi_\tau(v)$, or to zero if no such vertices exist. The group  $\mathbb{H} (\Gamma_\tau)$ is not finitely generated, but we can build two finitely generated groups by exploiting the $U$-action. Define
\begin{eqnarray*}
\mathbb{H}_{\mathrm{red}}(\Gamma_\tau)&:=&\mathrm{Coker}(U^n),\; \mathrm{for} \; \mathrm{large} \;n,\\
\widehat{\mathbb{H}}(\Gamma_\tau)&:=& \mathrm{Ker}(U)\oplus \mathrm{Coker}(U)[-1].
\end{eqnarray*}

\noindent It can be checked from the definition of a graded root that the first group is well defined.  The symbol $[-1]$ in the second equation indicates that we lower the degree of each homogeneous element of $\mathrm{Coker}(U)$ by one.

In his seminal work \cite{N}, N\'emethi 

\begin{itemize}
\item Constructed an abstract graded root $\Gamma_Y$ for every rational homology sphere $Y$ which bounds a special type of plumbed $4$-manifold $X$, called an almost rational plumbing, starting with the intersection form of $X$.

\item Showed that the Heegaard Floer homology group $HF^+(-Y)$ is isomorphic to $\mathbb{H}(\Gamma_Y)$ with a fixed degree shift.

\item Gave an algorithm for calculating a sequence $\tau_Y$ which generates $\Gamma_Y$, in a way similar to Laufer's method for finding Artin's fundamental cycle.

\item Explicitly calculated $\tau_Y$ for   all Seifert  rational homology spheres (with base orbifold $S^2$) in terms of their Seifert invariants.  
\end{itemize}

We now review N\'emethi's formulation of $\tau_Y$ for  Seifert homology spheres. For every positive integer $l$, let $(p_1,p_2,\dots,p_l)$ be an $l$-tuple of pairwise relatively prime  positive integers with $p_i\geq 2$ for all $i=1,\dots,l$. Denote by $Y=\Sigma(p_1,\ldots,p_l)$ the Seifert fibered space with base orbifold $S^2$ and normalized Seifert invariants  $(e_0,(p'_1,p_1),\ldots,(p'_l,p_l))$ 
where $e_0,p'_1,p'_2,\dots,p'_l$ are defined by
\begin{equation}\label{e:seifert}
e_0p_1p_2\cdots p_l+p_1'p_2\cdots p_l+p_1p_2'\cdots p_l+\cdots+p_1p_2\cdots p_l'=-1,
\end{equation}

\noindent with $0\leq p'_i\leq p_i-1$, for all $i=1,\dots,l$.  Such manifolds $Y$ are precisely the Seifert homology spheres.  The numbers $p_i$ are called the \emph{multiplicities} of the singular fibers of $Y$.  Permutations of multiplicities do not change the Seifert homology sphere. 

Let $\Delta_Y: \mathbb{N}\to \mathbb{Z}$ denote the function 
\begin{equation}\label{e:delta}
\Delta_Y (n)=  1 + |e_0|n - \sum_{i=1}^{l}\left \lceil \frac{np_i'}{p_i}  \right \rceil ,
\end{equation}
\noindent where $\lceil x \rceil$ is smallest integer greater than $x$. Let $\tau _Y$ be the unique solution of the difference equation
\begin{equation}\label{e:tau}
\tau _Y(n+1)- \tau _Y (n)= \Delta _Y (n) , \; \mathrm{with} \; \tau_Y (0)=0.
\end{equation}
It can be checked that $\tau_Y$ is non-decreasing after a finite index so it defines a graded root $\Gamma_{\tau_Y}$. 
\begin{theorem}[Nemethi, Section 11.13 of \cite{N}] \label{theo:Nemethi}
We have the following isomorphisms   of $\mathbb{Z}[U]$-modules
up to  an overall degree shift.
\begin{enumerate}
\item $HF^+(-Y)\cong \mathbb{H}(\Gamma_{\tau_Y})$,
\item $\HFred (-Y) \cong \mathbb{H}_{\mathrm{red}}(\Gamma_{\tau_Y})$,
\item $\HFhat (-Y) \cong \widehat{\mathbb{H}}(\Gamma_{\tau_Y})$.
\end{enumerate}
\end{theorem}

This theorem is sufficient for calculating the Heegaard Floer homology of a fixed Seifert homology sphere. On the other hand, one needs to develop a better understanding of the term $\Delta _Y$ in Equation (\ref{e:delta}),  in order to prove theorems regarding the Heegaard  Floer homology of infinite  families of Seifert homology spheres. The following result serves that purpose.  Recall that $Y=\Sigma(p_1,\ldots,p_l)$. Denote by  $G_Y$ the numerical semigroup  generated by $\displaystyle\frac{p_1p_2\dots p_l}{p_i}$ for $i=1,2,\dots,l$.  Define the constant
$$
N_Y=p_1p_2\cdots p_l\left ( \left (l-2 \right )-\sum_{i=1}^l\frac{1}{p_i}\right ).
$$

\begin{theorem} [Can-Karakurt, Theorem 4.1 of \cite{CK}] \label{theo:mainmore} $\;$

\begin{enumerate}
	\item \label{i:theo1} $N_Y$ is a positive integer, unless $l\leq 2$ or $l=3$ with $\{p_1,p_2,p_3\}=\{2,3,5\}$.
	\item \label{i:theo2} $\Delta_Y (n) \geq 0 $,  for all $n>N_Y$.
	\item \label{i:theo3} $\Delta_Y(n)=-\Delta_Y(N_Y-n)$,  for all $0\leq n\leq N_Y$.
		\item \label{i:theo4} $\Delta_Y(n)\in \mathbb{Z}\cap [-l+2,l-2]$, for all $n$ with $0\leq n \leq N_Y$.
	\item \label{i:theo5}For  $0\leq n \leq N_Y$, one has $\Delta_Y (n)\geq1$ if and only if $n \in G_Y$.
	\item \label{i:theo6} If $n\in G_Y$ is written in the form $\displaystyle n=p_1p_2\cdots p_l\left (\sum _{i=1}^l \frac{x_i}{p_i} \right )$ then $ \displaystyle \Delta_Y (n)=1 + \sum _{i=1}^l \left \lfloor \frac{x_i}{p_i} \right \rfloor$.
\end{enumerate}
\end{theorem}

\section{Abstract delta sequences and their morphisms}\label{s:absdel}

The discussion in the previous section provides an explicit method for the calculation of the graded roots and hence Heegaard Floer homology of Seifert homology spheres. On the other hand, it is a challenging question to give a closed formula for these objects.  Rather than attempting to find such a closed formula, we will develop some techniques to compare ranks of Heegaard Floer homology of two given Seifert homology spheres. That is why we shall define abstract delta sequences and study their morphisms in this section.  These objects essentially reduce the rank comparison problems to the existence of certain  maps between some combinatorial objects.

\subsection{Basic definitions and notation}\label{ss:def}

\begin{definition}\label{def:deltaseq}
A delta sequence is a pair $(X,\Delta)$ where
\begin{enumerate}
	\item \label{deltaseqit1} $X$ is a well--ordered finite set,
	\item \label{deltaseqit2}$\Delta : X\to \mathbb{Z}\setminus \{ 0 \}$ with $\Delta (z_0)>0$, where $z_0$ is the minimum of $X$.
\end{enumerate}
\end{definition}

We shall denote a delta sequence $(X,\Delta)$ simply by $\Delta$ if the set $X$ is clear from the context. Any delta sequence $\Delta$ naturally induces a graded root $\Gamma _\Delta$ as follows.  Write the ordered set as
$$X=\{z_0,z_1,\dots,z_{k-1} \},$$
\noindent with
$$z_0<z_1<\dots<z_{k-1}.$$
\noindent Then define a function
$$\tau_\Delta : \{0, 1,\dots,k \} \to \mathbb{Z}$$
\noindent using the recurrence relation
$$\tau_\Delta (n+1) -\tau_\Delta (n)= \Delta (z_{n}),$$ 
\noindent for $ n=0,\dots,k-1$, together with the initial condition
$$\tau _\Delta (0) =0.$$
  
\noindent The graded root $\Gamma_\Delta$ is the one induced by the function $\tau_\Delta$ as explained in Section \ref{s:Calc}. Conversely, every graded root comes from an abstract delta sequence. Of course many different delta sequences may induce the same graded root.  Let $X^+$ denote the well-ordered set $X\cup \{z^+ \}$ where $z^+>z$ for all $z\in X$. We find it convenient to think of the domain of $\tau_\Delta$ as being $X^+$ rather than $\{0,1,\dots,|X|\}$. Hence sometimes we abuse notation and write $\tau_\Delta (z)$ for $z\in X^+$ but actually mean $\tau_\Delta (n(z))$ where $n:X^+\to \{0,\dots,|X|\}$ is the order preserving enumeration of $X^+$.   With this convention, we have 
\begin{equation}\label{equ:tau}
\tau_\Delta(z)=\underset{w<z}{\sum_{w\in X}}\Delta(w),\;\mathrm{for}\;\mathrm{all}\;z\in X^+.
\end{equation} 

Recall that every graded root $\Gamma$ has an associated graded $\mathbb{Z}[U]$-module $\mathbb{H}(\Gamma)$ from which one obtains finitely generated free $\mathbb{Z}$-modules $\mathbb{H}_{\mathrm{red}} (\Gamma)$ and $\widehat{\mathbb{H}}(\Gamma)$, as described in Section~\ref{s:Calc}.  Our aim now is to give formulas calculating the ranks of these modules directly in terms of a delta sequence $\Delta$ inducing $\Gamma$. Let
$$ S_\Delta := \{ x \in X : \Delta (x) >0\},$$ 
$$ Q_\Delta := \{ y \in X : \Delta (y) <0\},$$ 
$$\kappa_\Delta := - \sum_{y\in Q_\Delta} \Delta (y).$$

\noindent Let $c_\Delta$ denote the number of $n$ such that $\Delta (z_n)<0$ and $\Delta (z_{n+1})>0$, for $n=1,\dots,k-2$. If $\Delta (z_{k-1})<0$, we also add one to $c_\Delta$.  

\begin{proposition}\label{prop:rank} We have
\begin{enumerate}
	\item \label{prop:rankitem1} $\mathrm{rank} (\widehat{\mathbb{H}}(\Gamma_\Delta)) = 2 c_\Delta + 1$,
	\item \label{prop:rankitem2}  $\displaystyle \mathrm{rank}(\mathbb{H}_{\mathrm{red}} (\Gamma_\Delta)) = \kappa _\Delta + \min _{z\in X} \tau _\Delta (z)$.

\end{enumerate}
\end{proposition}

\begin{proof}
For part (\ref{prop:rankitem1}), observe that the number of univalent vertices of $\Gamma_\Delta$ is exactly $c_\Delta + 1$. From the description of the $U$ action, it is now clear that $\rank (\ker U)=c_\Delta + 1$ and $\rank(\coker U )= c_\Delta$. Part (\ref {prop:rankitem2}) is \cite[Corollary 3.7]{N}. 
\end{proof}

\subsection{Operations on delta sequences} In this subsection, we discuss some methods to generate new delta sequences out of a given one. 
Henceforth for a given delta sequence $(X,\Delta)$, we shall reserve the symbols $\tau,\; \Gamma,\; S,\; Q,\; \kappa$, and $c$ to denote the objects introduced in Section \ref{ss:def}. Whenever a delta sequence admits a decoration, corresponding objects pick up the same decoration. Also the subscript $\Delta$ will be dropped for brevity. For example, if $(X_1,\Delta _1)$ is a delta sequence, then $S_1$ denotes the set of elements in $X_1$ for which $\Delta_1$ is positive. 

\begin{definition}
Let $(X,\Delta)$ be a delta sequence. A \emph{delta subsequence} of  $(X,\Delta)$ is a delta sequence $(X_1,\Delta _1)$ where  $X_1\subset X$ and $\displaystyle \Delta_1=\Delta|_{X_1}$.
\end{definition}

\begin{proposition}\label{prop:hatcomp}
If $(X_1,\Delta_1)$ is a delta subsequence of $(X,\Delta)$ then $$\mathrm{rank}(\widehat{\mathbb{H}}(\Gamma_1)) \leq \mathrm{rank}(\widehat{\mathbb{H}}(\Gamma)).$$
\end{proposition}
\begin{proof}
Clearly $c_1\leq c$. By Proposition \ref{prop:rank}, we are done.
\end{proof}

For the rank of $\mathbb{H}_{\mathrm{red}}(\Gamma_1)$, we shall prove a stronger inequality. To this end we introduce complementary subsequences. Suppose $(X_1,\Delta_1)$ is a delta subsequence of $(X,\Delta)$. From this we construct a new delta sequence by letting 
$X_2:=X\setminus X_1$, and $\displaystyle \Delta_2:= \Delta |_{X_2}$. The pair $(X_2,\Delta _2)$ may not satisfy Property (\ref{deltaseqit2}) of Definition \ref{def:deltaseq}, but we can modify it to become a delta sequence as follows: Let $x$ be the minimum of $S_2$. Remove all $y\in Q_2$ with $y\leq x$ from $X_2$. By abuse of notation we denote the resulting delta sequence by the same symbol, $(X_2,\Delta_2)$. This delta sequence is called the \emph{complementary delta subsequence} of $(X_1,\Delta_1)$ in $(X,\Delta)$.

\begin{proposition}\label{prop:compsub}
Let $(X,\Delta)$ be a delta sequence. Let $(X_1,\Delta_1)$ be a delta subsequence and $(X_2,\Delta_2)$  its complementary subsequence.  Then
$$ \mathrm{rank} (\mathbb{H}_{\mathrm{red}}(\Gamma_1)) + \mathrm{rank} (\mathbb{H}_{\mathrm{red}}(\Gamma_2)) \leq \mathrm{rank} (\mathbb{H}_{\mathrm{red}}(\Gamma)).$$
\end{proposition}

\begin{proof}
We use $\tau$, $\tau_1$, and $\tau_2$ for $\tau_\Delta$, $\tau_{\Delta_1}$, and $\tau_{\Delta_2}$ respectively.  First note that $S=S_1\cup S_2$, $Q\supseteq Q_1 \cup Q_2$, and
$$\kappa = \kappa_1 +\kappa _2 - \sum _{y\in Q\setminus (Q_1\cup Q_2)} \Delta (y).$$ 

\noindent Next we claim  that
\begin{equation} \label{eq:minineq}
\min \tau \geq \min \tau_1 + \min \tau _2 + \sum_{y \in Q\setminus(Q_1\cup Q_2)} \Delta (y).
\end{equation}

\noindent Let us first see why this inequality  finishes the proof. Rearranging the terms and adding $\kappa_1 +\kappa _2$ to both sides, we get
$$\min \tau + \kappa \geq \min \tau _1 + \min \tau _2 + \kappa _1 + \kappa_2.$$

\noindent By Proposition \ref{prop:rank} we are done.

\vspace{0.2cm}

To see why Inequality (\ref{eq:minineq}) holds, let $w_0\in X^+$ be an element where $\tau$ attains its minimum.   Then by Equation~(\ref{equ:tau})
\begin{eqnarray*}
\tau (w_0) &=& \underset{z< w_0}{\sum_{z\in X_1} }\Delta(z)+\underset{z< w_0}{\sum_{z\in X_2} }\Delta(z)+\underset{z< w_0}{\sum_{z\in X\setminus (X_1\cup X_2)} }\Delta(z)\\
 &=& \underset{z< w_0}{\sum_{z\in X_1} }\Delta(z)+\underset{z< w_0}{\sum_{z\in X_2} }\Delta(z)+ \underset{y<w_0}{\sum_{y\in Q\setminus (Q_1\cup Q_2)} }\Delta(y)\\
&\geq & \min  \tau_1 +  \min \tau_2 + \sum_{y\in Q\setminus (Q_1\cup Q_2)} \Delta(y),
\end{eqnarray*}
where the last inequality follows by noting that there exists $w_i\in X_i^+$ for $i=1,2$ such that
$$\underset{z< w_0}{\sum_{z\in X_i} }\Delta(z)=\tau(w_i).$$
\end{proof}

We shall define two more operations on delta sequences which are inverses of each other. Let  $(X,\Delta)$ be a delta sequence. Let $t$ be a positive integer  and $z\in X$ with $|\Delta (z)|\geq t$. From this we construct a new delta sequence $(X',\Delta ')$ as follows.  The set $X'$ is obtained by removing $z$ from $X$ and putting $t$ consecutive elements $z_1,\dots,z_t$  in its place. Now, choose non-zero integers $n_1,\dots,n_t$  each with the same sign as $\Delta(z)$, such that $n_1 + \dots + n_t = \Delta(z)$.  The new delta function $\Delta'$ agrees with $\Delta$ on $X\setminus \{z\}$ and it satisfies
$$\Delta ' (z_i) = n_i, \; \mathrm{for}\; i=1,\dots,t.$$
\noindent The delta sequence $(X',\Delta ')$ is called a \emph{refinement} of $(X,\Delta)$ at $z$. Conversely $(X,\Delta)$ is called the \emph{merge} of $(X',\Delta ')$ at $z_1,\dots,z_t$. 

\begin{proposition}\label{prop:refinementsmerges}
Refinements and merges do not change $\mathbb{H}_{\mathrm{red}}$ and $\widehat{\mathbb{H}}$.
\end{proposition}

\begin{proof}
This follows easily from the definitions.
\end{proof}

\subsection{Isomorphisms and embeddings}

We will define various kinds of maps between delta sequences and study their properties.

\begin{definition}\label{def:morph}
 A \emph{morphism} between delta sequences $(X_1,\Delta _1)$ and $(X_2,\Delta _2)$ is a map $\phi:X_1\to X_2$ with $\phi(S_1)\subseteq S_2$ and $\phi(Q_1)\subseteq Q_2$. A morphism $\phi$ is called an \emph{isomorphism} if it is an order preserving bijection satisfying $\Delta_2(\phi(z))=\Delta_1(z)$ for all $z\in X_1$. 
\end{definition}

\noindent Clearly isomorphic delta sequences induce the same graded root. Consequently we have the following.

\begin{proposition}\label{prop:iso}
If $(X_1,\Delta _1)$ and $(X_2,\Delta_2)$ are isomorphic delta sequences then  $\mathrm{rank}(\mathbb{H}_{\mathrm{red}}(\Gamma_1))=\mathrm{rank}(\mathbb{H}_{\mathrm{red}}(\Gamma_2)) $ and $\mathrm{rank}(\widehat{\mathbb{H}}(\Gamma_1))=\mathrm{rank}(\widehat{\mathbb{H}}(\Gamma_2))$.
\end{proposition}

In general, isomorphisms are difficult to construct. Most of the time we are content with isomorphisms up to refinements. We say that two delta sequences $(X_1,\Delta_1)$ and $(X_2,\Delta_2)$ are \emph{equivalent} if they are isomorphic after a sequence of refinements of each.

\begin{definition}\label{def:embed}
An \emph{embedding} of  $(X_1,\Delta_1)$ into $(X_2,\Delta_2)$ is a morphism $\phi: (X_1,\Delta_1)\to (X_2,\Delta _2)$ satisfying
\begin{enumerate}
\item \label{defn:emb1} For every $x\in S_1$ and $y\in Q_1$, we have $x < y$ if and only if $\phi (x) < \phi (y)$.
\item \label{defn:emb2} For every $z\in \phi(X_1),$ $\displaystyle |\Delta _2 (z)| \geq \sum _{w\in \phi^{-1}(z)} |\Delta_1(w)|$.
\end{enumerate}   
\end{definition}
 Property (\ref{defn:emb1}) of Definition~\ref{def:embed} says that the order of elements of $S_1$ relative to elements of $Q_1$ is preserved under an embedding.   Note also that unlike what the name might suggest,  an embedding of delta sequences need not be injective. The following result says we can achieve injectivity after we do appropriate modifications to the domain and target delta sequences. 

\begin{theorem}\label{theo:embedd}
If there is an embedding $\phi :(X_1,\Delta _1) \to (X_2,\Delta_2)$ then there are refinements $(X'_1,\Delta'_1)$ and $(X'_2,\Delta'_2)$ of $(X_1,\Delta _1)$ and $(X_2,\Delta_2)$ respectively such that $(X'_1,\Delta'_1)$ is isomorphic to a delta subsequence  of $(X'_2,\Delta'_2)$. In other words, if $\Delta_1$ embeds into $\Delta_2$, then $\Delta_1$ is equivalent to a delta subsequence of $\Delta _2$.

\end{theorem}

\begin{proof}
We first make $\phi$ injective using refinements on $(X_2,\Delta_2)$. Let $z\in\phi (X_1)$ such that $|\phi ^{-1}(z)| \geq 2$. Write $\phi^{-1}(z)=\{w_1,\dots,w_t\}$. We refine $\Delta_2$ at $z$ to $t+1$ elements $z_1,\dots,z_{t+1}$ such that $\Delta_2(z_n)=\Delta_1(w_n)$ for all $n=1,\dots,t$. This is possible by Property (\ref{defn:emb2}) of Definition \ref{def:embed}. We extend $\phi$ to this refinement by sending $w_n$ to $z_n$ for all $n=1,\dots,t$. Repeating this process for every $z\in \phi (X_1)$ with $|\phi^{-1}(z)| \geq 2$, we get a refined delta sequence $(X'_2,\Delta'_2)$, and an injective embedding $\phi':(X_1,\Delta_1)\to (X'_2,\Delta'_2)$.  If necessary, we can refine $(X'_2,\Delta'_2)$ further to achieve that $\Delta'_2(\phi'(w))=\Delta  _1(w)$ for all $w\in X_1$; this is again guaranteed by Property (\ref{defn:emb2}) of Definition \ref{def:embed}.

\vspace{0.2cm}

The morphism $\phi '$ may not be order preserving, so we need to modify it to have this property.  Since $\phi'$ is an embedding, the relative positions of elements of $S_1$ do not change with respect to the elements of $Q_1$. On the other hand, $\phi'$ can rearrange some consecutive elements in $S_1$ (and respectively in $Q_1$). We can permute the order of these elements by a sequence of refinements and merges and re-defining $\phi'$ accordingly. To see that this is possible, suppose $x,\tilde{x}\in S_1$ with $x<\tilde{x}$ but $\phi ' (\tilde{x})<\phi '(x)$. Suppose also that there is no $w\in X_1$ with $x<w<\tilde{x}$. Merge $x$ and $\tilde{x}$ together and subsequently refine the resulting element into $x'$, $\tilde{x}'$ with $x'<\tilde{x}'$ to obtain a new delta sequence $(X'_1,\Delta'_1)$ so that $\Delta '_1(x')=\Delta_1(\tilde{x})$ and $\Delta ' _1(\tilde{x}')=\Delta_1(x)$. Define a morphism $\phi'':(X'_1,\Delta'_1) \to (X'_2,\Delta'_2)$ so that $\phi '' (x')=\phi' (\tilde{x})$, $\phi''(\tilde{x}')=\phi'(x)$, and $\phi''$ agrees with $\phi'$ otherwise. Hence $\phi ''$ is still an injective embedding and now it preserves the order of the elements $x'$ and $\tilde{x}'$. Repeating this process for every pair in $S_1$ (and respectively $Q_1$) where the order preserving fails, eventually we make $\phi ''$ order preserving, since transpositions generate the whole permutation group.  We use the same symbols $\phi'': (X'_1,\Delta'_1) \to (X'_2,\Delta'_2)$ to denote the resulting morphism after making all necessary adjustments.  

\vspace{0.2cm}

Now $\phi''$ becomes an order preserving injective morphism with $\Delta'_2(\phi''(w))=\Delta ' _1(w)$ for all $x\in X'_1$. Hence it is an  isomorphism onto its image.
\end{proof}

In light of Theorem~\ref{theo:embedd}, given an embedding of $\Delta_1$ into $\Delta$, we can find a delta subsequence of $\Delta$ equivalent to $\Delta_1$.  By abuse of terminology, we will refer to the corresponding complementary subsequence as the complementary delta subsequence of $\Delta_1$.  Note that the complementary subsequence implicitly depends on the choice of embedding.  
 
\begin{corollary}\label{cor:embed}
Suppose $\Delta_1$ embeds into $\Delta$ and let $\Delta_2$ denote the complementary delta subsequence.  Then
\begin{enumerate}
	\item $\mathrm{rank}(\widehat{\mathbb{H}}(\Gamma_1)) \leq \mathrm{rank}(\widehat{\mathbb{H}}(\Gamma))$,
	\item $\mathrm{rank} (\mathbb{H}_{\mathrm{red}}(\Gamma _1)) + \mathrm{rank} (\mathbb{H}_{\mathrm{red}}(\Gamma _2)) \leq \mathrm{rank} (\mathbb{H}_{\mathrm{red}}(\Gamma ))$.
\end{enumerate}

\end{corollary}

\begin{proof}
The first inequality (second inequality respectively) follows from combining Proposition~\ref{prop:hatcomp} (respectively Proposition~\ref{prop:compsub}) with Proposition~\ref{prop:refinementsmerges}, Proposition~\ref{prop:iso}, and  Theorem~\ref{theo:embedd}.
\end{proof}

\begin{corollary}\label{cor:disjointemb}
If $(X_1,\Delta _1),\dots,(X_p,\Delta_p)$  can be disjointly embedded into $(X,\Delta)$ then 
$$\sum_{j=1}^p \mathrm{rank}(\mathbb{H}_{\mathrm{red}}(\Gamma_j)) \leq \mathrm{rank}( \mathbb{H}_{\mathrm{red}}(\Gamma)). $$
\end{corollary}

\begin{proof}
The disjointness ensures that $(X_1,\Delta _1),\dots,(X_{p-1},\Delta_{p-1})$ embed into the complementary subsequence of $(X_p, \Delta_p)$. Now use Corollary \ref{cor:embed} and do induction on $p$.   
\end{proof}
\subsection{Right-veering maps and immersions} We now relax the condition on the preservation of  orderings of our morphisms. Though we can still prove rank inequalities  for $\mathbb{H}_{\mathrm{red}}$ under these new kinds of morphisms, we need to give up the rank inequality for $\widehat{\mathbb{H}}$.

\begin{definition}\label{def:immn}
A \emph{right-veering morphism} between delta sequences $(X_1,\Delta_1)$ and $(X_2,\Delta_2)$ is a bijective morphism $\phi :(X_1,\Delta_1) \to (X_2,\Delta_2)$ such that
\begin{enumerate}
\item \label{def:immn1}Both of the maps $\phi|_{S_1}$ and $\phi|_{Q_1}$ are order preserving,
\item \label{def:immn2}For all $x\in S_1$, $y\in Q_1$ with $x< y$, we have $\phi(x)< \phi (y)$,
\item \label{def:immn3}$\Delta_1(z)=\Delta _2(\phi (z))$ for all $z\in X_1$. 
\end{enumerate}
\end{definition}

Note that a right-veering morphism $\phi$ is almost  the same thing as an isomorphism except that one may have elements $x\in S_1$, $y\in Q_1$ with $y\leq x$ and $\phi(x)\leq \phi(y)$. In other words $\phi$ ``moves'' the elements of $Q_1$ to the ``right'' of elements of $S_1$.

\begin{proposition}\label{prop:rightveer}
If there is a right-veering morphism $\phi:(X_1,\Delta_1)\to(X_2,\Delta_2)$ then 
$$\mathrm{rank} (\mathbb{H}_{\mathrm{red}}(\Gamma _1))\leq \mathrm{rank} (\mathbb{H}_{\mathrm{red}}(\Gamma _2)).$$
\end{proposition}

\begin{proof}
First observe that $\kappa_1=\kappa_2$ by Property (\ref{def:immn3}) of Definition \ref{def:immn} and the fact that $\phi$ is a bijective morphism. Hence by Proposition \ref{prop:rank}, it suffices to show that  $\min \tau_1 \leq \min \tau_2$. To do this we shall show that given any $z_2\in X_2^+$ there exists $z_1\in X_1^+$ such that $\tau_1(z_1)\leq \tau_2(z_2)$.  Let $z_2\in X_2^+$ be given. If $z_2$ is the maximal element $z_2^+\in X_2^+$ then we choose $z_1$ to be the maximal element $z_1^+\in X_1^+$, and we have $\tau(z_1)=\tau(z_2)$ by Property (\ref{def:immn3}) and the fact that  $\phi$ is a bijection. Otherwise $z_2\in X_2$, and we  choose $z_1:=\phi^{-1}(z_2)$.

\vspace{0.2cm}

Using Properties (\ref{def:immn1}) and (\ref{def:immn2}) in Definition~\ref{def:immn}, we see that for all $x_1\in S_1$, if $x_1 \leq z_1$ then $\phi(x_1) \leq z_2$.  Therefore, 
\begin{equation}\label{eqn:rv1}
\underset{x_1< z_1}{\sum_{x_1\in S_1} }\Delta_1(x_1) = \underset{x_1< z_1}{\sum_{x_1\in S_1} }\Delta_2(\phi (x_1)) \leq \underset{x_2< z_2}{\sum_{x_2\in S_2} }\Delta_2(x_2),  
\end{equation}
where the last inequality is in fact an equality if $z_2\in S_2$. Similarly we see that for all $y_2\in Q_2$, $\phi^{-1}(y_2)\leq z_1$ if $y_2\leq z_2$. Hence we have 
\begin{equation}\label{eqn:rv2}
\underset{y_1< z_1}{\sum_{y_1\in Q_1} }\Delta_1(y_1) = \underset{y_1< z_1}{\sum_{y_1\in Q_1} }\Delta_2(\phi (x_1)) \leq \underset{y_2< z_2}{\sum_{y_2\in Q_2} }\Delta_2(y_2),
\end{equation}
where the last inequality is in fact an equality if $z_2\in Q_2$. Combining Equations (\ref{eqn:rv1}) and (\ref{eqn:rv2}), we see 
\begin{eqnarray*}
\tau_1(z_1) &=& \underset{x_1< z_1}{\sum_{x_1\in S_1} }\Delta_1(x_1)+\underset{y_1< z_1}{\sum_{y_1\in Q_1} }\Delta_1(y_1)  \\
&\leq & \underset{x_2< z_2}{\sum_{x_2\in S_2} }\Delta_2(x_2) +\underset{y_2< z_2}{\sum_{y_2\in Q_2} }\Delta_2(y_2) \\
&=& \tau_2(z_2).
\end{eqnarray*}
\end{proof}

\begin{definition}\label{def:immers}
A morphism $\phi :(X_1,\Delta _1) \to (X_2,\Delta _2)$ is called a \emph{semi-immersion} if for every $x\in S_1$ and $y\in Q_1$ with $x< y$, we have $\phi(x) < \phi (y)$.  A semi-immersion is called an \emph{immersion}, if  it satisfies $\displaystyle |\Delta_2 (z) |\geq  \sum _{x\in \phi ^{-1}(z)} |\Delta _1 (x)|$ for all $z\in \phi (X_1)$.   
\end{definition}

\begin{theorem}\label{theo:immerssub}
Suppose there exists an immersion from $(X_1,\Delta _1)$ to $(X_2,\Delta _2)$. Then the following exist.
\begin{itemize}
\item A refinement $(X'_1,\Delta'_1)$ of $(X_1,\Delta _1)$.
\item A refinement $(X'_2,\Delta'_2)$ of $(X_2,\Delta _2)$.
\item A delta subsequence  $(\widetilde{X}_2,\widetilde{\Delta} _2)$ of $(X'_2,\Delta'_2)$.
\item A right-veering morphism $\widetilde{\phi} :(X'_1,\Delta'_1) \to (\widetilde{X}_2,\widetilde{\Delta} _2)$.
\end{itemize}
\end{theorem}

\begin{proof}
Similar to the proof of Theorem \ref{theo:embedd}.
\end{proof}

\begin{corollary}\label{cor:rankimmers}
If $(X_1,\Delta_1)$ can be immersed into $(X_2,\Delta_2)$ then 
$$\mathrm{rank} (\mathbb{H}_{\mathrm{red}}(\Gamma _1))\leq \mathrm{rank} (\mathbb{H}_{\mathrm{red}}(\Gamma _2)).$$
\end{corollary}

\begin{proof}
This follows from Proposition  \ref {prop:compsub}, Proposition \ref{prop:rightveer}, and Theorem \ref{theo:immerssub}.  
\end{proof}

\subsection{Well-behaved semi-immersions}\label{subsec:wel}
Let $\phi:(X_1,\Delta _1)\to (X_2,\Delta _2)$ be a one-to-one semi-immersion. For $z\in X_1$, the \emph{defect} of $z$ is the number
$$d_\phi(z):=|\Delta_1 (z)|-|\Delta_2(\phi(z))|.$$ 
An element $z\in X_1$ is called
\begin{itemize}
\item a \emph{bad point} if it has positive defect,
\item a \emph{good point} if it has negative defect,
\item a \emph{neutral point} if it has zero defect.
\end{itemize}

\noindent We shall call the images of bad (respectively good, neutral) points, \emph{bad} (respectively \emph{good, neutral}) \emph{values}.

\vspace{0.2cm}

Of course a one-to-one semi-immersion $\phi$ is an immersion if it has no bad points. We shall investigate how much we can relax this condition to still obtain rank inequalities for $\mathbb{H}_{\mathrm{red}}$ under one-to-one semi-immersions. Denote the set of good points and the set of bad points by $G(\phi)$ and $B(\phi)$  respectively.

\begin{definition}\label{d:control} A \emph{control function} is an injection $\theta:B(\phi)\to G(\phi)$ satisfying:
\begin{enumerate}
\item \label{def:cont1} $z\in S_1$ if and only if $\theta(z)\in S_1$,
\item \label{def:cont2} $\theta (z) <z $ if $z\in S_1 \cap B(\phi)$, and $\theta (z) >z $ if $z\in Q_1 \cap B(\phi)$,
\item \label{def:cont3} $|d_\phi(z)|\leq |d_\phi(\theta(z))|$ for all $z\in B(\phi)$. 
\end{enumerate}
\end{definition}

\noindent We will use control functions to ``fix'' defects of bad points with merges and refinements using corresponding good points. Call a one-to-one semi-immersion \emph{well-behaved} if it admits a control function. 

\begin{theorem}\label{theo:rankwelbeh}
Let $\phi : (X_1,\Delta_1) \to (X_2,\Delta_2)$ be delta sequences. Suppose there is a well-behaved one-to-one semi-immersion $\phi:(X_1,\Delta_1) \to (X_2,\Delta_2)$. Then there exist refinements $(X'_1,\Delta '_1)$ and $(X'_2,\Delta'_2)$ of $(X_1,\Delta_1)$ and $(X_2,\Delta_2)$ respectively, and a one-to-one immersion $\phi ' : (X'_1,\Delta '_1)\to (X'_2,\Delta'_2)$.
\end{theorem}

\begin{proof}
Let $\theta:B(\phi)\to G(\phi)$ be a control function. We define a refinement $(X'_1,\Delta'_1)$ of $(X_1,\Delta_1)$ at each $b\in B(\phi)$ by replacing it with a consecutive pair $b_1$, $b_2$ with the additional requirement that $\Delta_1'(b_1)=\Delta_2(\phi(b))$. Note that this automatically ensures $|\Delta_1'(b_2)|=|d_\phi(b)|$.   Similarly define a refinement $(X'_2,\Delta'_2)$ of $(X_2,\Delta_2)$ at each $g=\phi(\theta(b))$ for $b\in B(\phi)$, by replacing $g$ with a consecutive pair $g_1$, $g_2$ with the requirement that $\Delta_2'(g_1)=\Delta_1(\theta(b))$. Again this implies $|\Delta_2'(g_2)|=|d_\phi(\theta(b))|$.

\vspace{0.2cm}

Now define a map $\phi':(X_1',\Delta_1')\to (X_2',\Delta_2')$ by 
\begin{equation}\label{e:refine}
\phi'(z) =
\left\{
\begin{array} {rl}
\phi(z) &\text{ if }\; z \in X_1 \setminus (B(\phi)\cup \theta (B(\phi))), \\[2pt]
\phi(b) & \text{ if }\; z = b_1 \; \text{for} \; \text{some} \; b\in B(\phi), \\[2pt]
g_1 & \text{ if }\; z = \theta(b)  \; \text{for} \; \text{some} \; b\in B(\phi) \; \text{and} \; g=\phi(\theta(b)), \\[2pt]
g_2 & \text{ if }\;  z = b_2  \; \text{for} \; \text{some} \; b\in B(\phi) \; \text{and} \; g=\phi(\theta(b)).
\end{array} \right.
\end{equation}
We check that $\phi'$ satisfies all the claimed properties. Clearly $\phi'$ is injective if $\phi$ is injective. Property (\ref{def:cont1}) of control functions ensures that $\phi'$ is a morphism. Property (\ref{def:cont2}) of being a control function implies $\phi'$ is a semi-immersion.  It is left to show that $\phi '$ has no bad points. Clearly no element of $X_1\setminus (B(\phi)\cup \theta (B(\phi)))$ is bad. Suppose $b\in B(\phi)$, we will show none of $b_1,b_2$ and $\theta (b)$ are bad. Equation (\ref{e:refine}) shows that under $\phi'$, $b_1$ and $\theta(b)$ are neutral.  Finally by Property (\ref{def:cont3}) of control functions, we have
$$|\Delta_1'(b_2)|=|d_\phi(b)|\leq |d_\phi(\theta(b))| =|\Delta_2 '(g_2)|.$$
\noindent Hence $b_2$ is not a bad point.
\end{proof}

\begin{corollary}\label{cor:ranksemi}
If there is a well-behaved one-to-one semi immersion $\phi:(X_1,\Delta_1)\to (X_2,\Delta_2)$ then
$$\mathrm{rank} (\mathbb{H}_{\mathrm{red}}(\Gamma _1))\leq \mathrm{rank} (\mathbb{H}_{\mathrm{red}}(\Gamma _2)).$$
\end{corollary}

\begin{proof}
This follows from  Proposition~\ref{prop:iso}, Corollary~\ref{cor:rankimmers},  and Theorem~\ref{theo:rankwelbeh}.
\end{proof}

\subsection{Delta sequences of Seifert homology spheres}\label{ss:deltaforseif}

It is our goal to reduce the comparison of the Heegaard Floer homology of Seifert homology  spheres to the comparison of the associated  delta sequences. In order to use the techniques developed in this section we need to see that the graded root of a Seifert homology sphere is naturally induced by an abstract delta sequence. 

\vspace{0.2cm}

Let $Y=\Sigma (p_1,\dots,p_l)$ be a Seifert homology sphere, and $\Delta _Y$  its delta function. We shall use Theorem \ref{theo:mainmore} and  the notation introduced there. In particular $G_Y$ denotes the numerical semigroup generated by 
$$\frac{p_1\cdots p_l}{p_i}\quad \text{for $i=1,\dots,l$}.$$
The graded root and hence the Heegaard Floer homology of $Y$ is completely determined by the values of $\Delta _Y$ on the interval $[0,N_Y]$ because it becomes non-negative afterward. On the  interval $[0,N_Y]$,  $\Delta_Y$ takes all its positive values on the set  $S_Y:=G_Y\cap [0,N_Y]$, and it takes all its negative values on the set $Q_Y:=\{N_Y-x: x \in S_Y \} $. Note that $S_Y\cap Q_Y = \emptyset$. Let $X_Y:=S_Y\cup Q_Y$.

\begin{proposition}
For any Seifert homology sphere $Y$ which is not $S^3$ or $\Sigma(2,3,5)$, the pair $(X_Y,\Delta_Y|_{X_Y})$ is an abstract delta sequence in the sense of Definition \ref{def:deltaseq}. The graded root induced by this abstract delta sequence is the same as the graded root of $Y$. 
\end{proposition}
\begin{proof}
It is clear that Property (\ref{deltaseqit1}) of Definition \ref{def:deltaseq} is satisfied.  Property (\ref{deltaseqit2}) holds, since $\Delta_Y(0) = 1$.  Therefore, $(X_Y,\Delta_Y|_{X_Y})$ is an abstract delta sequence.  For the second claim, we simply observe that the zeros of $\Delta_Y$ and the values of $\Delta_Y$ after $N_Y$ do not affect the graded root.  
\end{proof}

In particular, we have that $\mathbb{H}_{\mathrm{red}}(\Gamma_{\Delta_Y})$ (respectively $\widehat{\mathbb{H}}(\Gamma_{\Delta_Y})$) is isomorphic to $\HFred (-Y)$ (respectively $\widehat{HF}(-Y)$) by Theorem \ref{theo:Nemethi}. The duality of Heegaard Floer homology under orientation reversal \cite[Proposition 2.5]{OS5} implies   $\rank \mathbb{H}_{\mathrm{red}}(\Gamma_{\Delta_Y})= \rank \HFred (Y)$ and $\rank \widehat{\mathbb{H}}(\Gamma_{\Delta_Y}) =\rank \widehat{HF}(Y)$.

\vspace{0.2cm}

Next we discuss a practical method to generate semi-immersions between delta sequences of Seifert homology spheres from maps between their semigroups.

\begin{definition}\label{def:tight}
Suppose $Y$ and $Y'$ are Seifert homology spheres.  If $\phi : G_Y \to G_{Y'}$ is a one-to-one function such that 

\begin{enumerate}
\item \label{d:tight1} $\phi(x) \geq x$,  
\item \label{d:tight2} $\phi(x) - x \leq \frac{N_{Y'} - N_{Y}}{2}$,  
\end{enumerate}
for every $x\in G_Y$, then $\phi$ is called {\em rigid}.  
\end{definition}

\begin{lemma}\label{lem:tightimpliesrightveering}
Let $\phi:G_Y \to G_{Y'}$ be rigid.  Then, $\phi$ naturally determines a one-to-one semi-immersion on the associated abstract delta sequences by restricting $\phi$ to $S_Y$ and by extending $\phi$ to $Q_Y$ by $\phi(N_Y - x) = N_{Y'} - \phi(x)$.
\end{lemma}
\begin{proof}
Let us first verify that $\phi$ defines a morphism from $\Delta_Y$ to $\Delta_{Y'}$. Observe that Properties (\ref{d:tight1}) and (\ref{d:tight2}) in Definition \ref{def:tight} force that $N_Y \leq N_{Y'}$. If $x\in S_Y$, then $x< N_Y$. Use Property (\ref{d:tight2}) in Definition \ref{def:tight} to see that
$$\phi(x) \leq \frac{N_{Y'}}{2} + x - \frac{N_Y}{2} < \frac{N_{Y'}+N_Y}{2} \leq N_{Y'}.$$
\noindent Hence $\phi(S_Y)\subseteq S_{Y'}$. Furthermore,  $\phi(N_Y-x)=N_{Y'}-\phi(x)$ by definition. Since  $\phi(x)\geq 0$, we have $\phi(N_Y-x) >0$. Therefore $\phi(Q_Y)\subseteq Q_{Y'}$.  

\vspace{0.2cm}

The last thing to check is that if $x_1 < N_Y-x_2$ for $x_1,x_2\in S_{Y}$ then $\phi (x_1) < N_{Y'}-\phi(x_2)$. Again it follows from Property (\ref{d:tight2}) in Definition~\ref{def:tight} that 
$$\phi(x_1) -x_1 +\phi(x_2) -x_2 \leq N_{Y'}-N_Y.$$
\noindent Hence
$$ \phi (x_1) + \phi(x_2) \leq N_{Y'}- (N_{Y}-x_1-x_2) <N_{Y'},$$
\noindent where the last inequality follows from $x_1<N_Y-x_2$.
\end{proof}

\subsection{Guide} The following diagram shows where various types of morphisms are defined and where they are going to be used in our argument. 
\noindent

\begin{equation*}
\begindc{\commdiag}[10]
\obj(20,60)[Iso]{
\parbox[b][3em][c]{0.4\textwidth}{Isomorphisms: Definition \ref{def:morph}\\ Rank equality by Proposition \ref{prop:iso}} \vspace{0.3cm}}
\obj(35,50)[Rig]{
\parbox[b][3em][c]{0.4\textwidth}{\vspace{0.2cm} Right-veerings: Definition \ref{def:immn}\\ Rank inequality by Proposition \ref{prop:rightveer}} }
\obj(12,50)[Emb]{
\parbox[b][3em][c]{0.4\textwidth}{\vspace{0.2cm} Embeddings: Definition \ref{def:embed}\\ Rank inequality by Corollary \ref{cor:embed} } }
\obj(9,44)[Dis]{
\parbox[b][3em][c]{0.4\textwidth}{\vspace{0.5cm} Disjoint Embeddings \\ Rank inequality by Corollary \ref{cor:disjointemb}} }
\obj(24, 40)[Imm]{
\parbox[b][3em][c]{0.5\textwidth}{\vspace{0.5cm} Immersions: Definition \ref{def:immers}\\ Rank inequality by Corollary \ref{cor:rankimmers}} }
\obj(35, 35)[Wel]{
\parbox[b][3em][c]{0.5\textwidth}{\vspace{0.5cm} Well-behaved semi-immersions: Section \ref{subsec:wel} \\ Rank inequality by Corollary \ref{cor:ranksemi}} }
\obj(5, 34)[Bra]{
\parbox[b][3em][c]{0.3\textwidth}{\vspace{0.5cm} Rank inequality for \\branched covers: \\ Theorem~\ref{thm:fiberbranched}, Section \ref{sec:bra}} }
\obj(15, 29)[Par]{
\parbox[b][3em][c]{0.3\textwidth}{\vspace{0.5cm} Partial order \\ rank inequality: \\ Theorem~\ref{thm:monotonicity}, Section \ref{sec:par}} }
\obj(26, 22)[Sem]{
\parbox[b][3em][c]{0.4\textwidth}{\vspace{0.5cm} Semi-immersions:\\ Definition \ref{def:immers} \\ No rank inequality} }
\obj(42, 24)[Pin]{
\parbox[b][3em][c]{0.4\textwidth}{\vspace{0.5cm} Rank inequality for\\ vertical pinches:\\ Theorem~\ref{thm:pinch}, Section \ref{sec:pin}} }

\mor(23,59)(33,51){}[1,0]
\mor(17,59)(12,51){}[1,0]
\mor(8,49)(6,45){}[1,0]
\mor(15,49)(18,40){}[1,0]
\mor(32,49)(22,40){}[1,0]
\mor(33,49)(36,35){}[1,0]
\mor(7,43)(3,35){}[1,0]
\mor(17,39)(12,30){}[1,0]
\mor(21,39)(21,23){}[1,0]
\mor(30,34)(21,23){}[1,0]
\mor(33,34)(37,25){}[1,0]
\enddc
\end{equation*}

\section{Proof of the Rank inequality for Branched Covers}\label{sec:bra}
We start with a remark about the content of the proof.
\begin{remark}
For $S^3$ or $\Sigma(2,3,5)$, we have $\HFred=0$. Hence Theorems \ref{thm:fiberbranched} and \ref{thm:monotonicity} trivially hold if one of the Seifert homology spheres being considered is either $S^3$ or $\Sigma(2,3,5)$. For this reason, we will not consider these manifolds in the proofs of Theorems \ref{thm:fiberbranched} and \ref{thm:monotonicity}. Also note that  $S^3$ or $\Sigma(2,3,5)$ cannot arise in the statement of Theorem \ref{thm:pinch}.
\end{remark}

\begin{proof}[Proof of Theorem \ref{thm:fiberbranched}]
  
Let $Y = \Sigma(p_1,\ldots,p_l)$ and $Y' = \Sigma(p_1,\ldots,p_{l-1},n p_l)$. Consider the associated abstract delta sequences $(X_Y,\Delta_Y)$ and $(X_{Y'},\Delta_{Y'})$ as described in Section \ref{ss:deltaforseif}. Our aim is to show that $n \rank \mathbb{H}_{\mathrm{red}}(\Gamma_{\Delta_Y})  \leq \rank \mathbb{H}_{\mathrm{red}}(\Gamma_{\Delta_{Y'}})$, and appeal to Theorem~\ref{theo:Nemethi}. By Corollary~\ref{cor:disjointemb}, it suffices to find $n$ disjoint embeddings of $\Delta_Y$ into $\Delta_{Y'}$.  
For $0 \leq k \leq n - 1$, define 
\[
\phi_k : X_Y \to X_{Y'}, z \mapsto n z + k p_1 \cdots p_{l-1}.   
\] 
We first need to see that each $\phi_k$ is an embedding. In particular we must show that each $\phi_k$ is a morphism. It is clear that $\phi_k$ takes elements of $G_Y$ to $G_{Y'}$.  Let $z \in X_Y$.  We observe that 
\begin{align*}
\phi_k(z) \leq \phi_{n-1}(z) = n z + (n - 1)p_1 \cdots p_{l-1} \leq n N_Y + (n - 1) p_1 \cdots p_{l-1}  = N_{Y'}
\end{align*}
This implies that $\phi_k(S_Y) \subset S_{Y'}$.  
Now, if $y = N_{Y} - x \in Q_{Y}$, where $x \in S_Y$, we have 
\begin{align}\label{eq:phik}
\phi_k(y) = n y + k p_1 \cdots p_{l-1} & = n N_Y + k p_1 \cdots p_{l-1} - nx \\
& = N_{Y'} - (n - 1)p_1 \cdots p_{l-1} + k p_1 \cdots p_{l-1} - n x \\
\label{eq:phik2}& = N_{Y'} - (n - 1 -k)p_1 \cdots p_{l-1} - n x.     
\end{align}
Since $0 \leq k \leq n - 1$, we have that $(n - 1 - k) p_1 \cdots p_{l-1} + n x=\phi_{n-1-k}(x) \in S_{Y'}$.  Therefore, $\phi_k(y) \in Q_{Y'}$.  Thus,  each $\phi_k$ is a morphism between the delta sequences $\Delta_Y$  and $\Delta_{Y'}$.  Furthermore, it is clear that each $\phi_k$ is order preserving (it is of the form $ax + b$ with $a>0$).  Therefore, Property (\ref{defn:emb1}) in Definition \ref{def:embed} is also satisfied.  

\vspace{0.2cm}

To complete showing that the $\phi_k$ are embeddings, it remains to check that $|\Delta_{Y'}(\phi_k(z))| \geq |\Delta_Y(z)|$ for all $z \in X_Y$, since the $\phi_k$ are injective.  The following claim establishes this.

\begin{claim}\label{claim:delta}
For each $z \in X_Y$, $\Delta_{Y'}(\phi_k(z)) = \Delta_Y(z)$.  
\end{claim} 
\begin{proof}[Proof of Claim~\ref{claim:delta}]
Recall $X_Y=S_Y\cup Q_Y$. We begin with the case of $x  \in S_Y$.  Express $x$ by 
\[
x = p_1\cdots p_l \sum^l_{i=1} \frac{a_i}{p_i}.
\]
By Theorem \ref{theo:mainmore} part \eqref{i:theo6}, 
\begin{align*}
\Delta_{Y'}(\phi_k(x)) &= \Delta_{Y'}(n x + k p_1 \cdots p_{l-1}) \\
&= \sum^{l-1}_{i=1} \left \lfloor \frac{a_i}{p_i} \right \rfloor + \left \lfloor \frac{na_l + k}{np_l} \right \rfloor +1.
\end{align*}
Since $0 \leq k < n$, it is straightforward to check that   
\[
\left \lfloor \frac{na_l + k}{np_l} \right \rfloor = \left \lfloor \frac{a_l}{p_l} \right \rfloor.
\]
Therefore, 
\begin{equation}\label{eqn:branchedembedding}
\Delta_{Y'}(\phi_k(x)) = \sum^{l}_{i=1} \left \lfloor \frac{a_i}{p_i} \right \rfloor +1= \Delta_Y(x).   
\end{equation}  

\noindent On the other hand, if $y = N_Y - x \in Q_Y$, then by Equations (\ref{eq:phik})-(\ref{eqn:branchedembedding}) and Theorem \ref{theo:mainmore}  part \eqref{i:theo3},
\begin{align*}
\Delta_{Y'}(\phi_k(y)) &= \Delta_{Y'}(N_{Y'} - (n - 1 -k)p_1 \cdots p_{l-1} - n x) \\
&= - \Delta_{Y'}(n x + (n - 1 -k)p_1 \cdots p_{l-1} ) \\
&= -\Delta_Y(x) \\
&= \Delta_Y(y).    
\end{align*}
\end{proof}

It remains to see that for $i \neq j$, the images of $\phi_i$ and $\phi_j$ are disjoint.  Suppose that $\phi_i(z_1) = \phi_j(z_2)$ for some $z_1, z_2 \in X_Y$.  In this case, either both $z_1,z_2 \in S_Y$ or $z_1,z_2 \in Q_Y$.  First, if $z_1,z_2 \in S_Y$, then 
\[
\phi_i(z_1) \equiv ip_1\cdots p_{l-1} (\text{mod } n) \; \text{ and } \; \phi_j(z_2) \equiv j p_1 \cdots p_{l-1} (\text{mod } n).
\]
Therefore, 
\[
i p_1 \cdots p_{l-1} \equiv j p_1 \cdots p_{l-1} (\text{mod } n).
\]
Therefore, $i \equiv j (\text{mod } n)$ since $\gcd (p_k,n) = 1$ for all $1 \leq k \leq l-1$.  However, since $0 \leq i, j \leq n-1$, we must have that $i = j$.  A similar argument applies to the case of $z_1,z_2 \in Q_Y$.  
\end{proof}

The proof of Theorem~\ref{thm:fiberbranched} also implies a weaker rank inequality for $\HFhat$. 

\begin{proposition}\label{prop:fiberbranchedhat} We have
$\displaystyle
 \rank \HFhat(\Sigma(p_1,\ldots,p_l)) \leq \rank \HFhat(\Sigma(p_1,\ldots,p_{l-1},n p_l)).  
$
\end{proposition}

\begin{proof}
Let $Y = \Sigma(p_1,\ldots,p_l)$ and $Y' = \Sigma(p_1,\ldots,p_{l-1},n p_l)$ as before. In the proof of Theorem~\ref{thm:fiberbranched} we constructed an embedding of $\Delta_Y$ into $\Delta_{Y'}$. Hence Corollary~\ref{cor:embed} and Theorem~\ref{theo:Nemethi} finish the proof.
\end{proof}

\section{Proof of the Partial Order Inequality}\label{sec:par}
Our purpose is to prove Theorem \ref{thm:monotonicity}. First we prove a lemma about our numerical semigroups. Let $p_1,\dots,p_l$ be pairwise relatively prime positive integers with $p_i\geq 2$ for all $i=1,\dots,l$. Let $G$ be the numerical semigroup generated by $\displaystyle\frac{p_1\cdots p_l}{p_i}$ for $i=1,\dots,l$.  

\begin{lemma}\label{lem:normalform}
Every element $n$ of $G$ can be uniquely written as 
\begin{equation}\label{eq:normalform}
n=p_1\cdots p_l\left ( k+\sum _{i=1}^l\frac{x_i}{p_i} \right ),
\end{equation}

\noindent for some integers $k\geq 0$ and $0\leq x_i <p_i$, for all $i=1,\dots, l$.
\end{lemma}

\begin{proof}
First we show that the form (\ref{eq:normalform}) exists for every element of $G$. Let $n\in G$, and write it as a linear combination of the generators
$$n=p_1\cdots p_l \sum _{i=1}^l \frac{a_i}{p_i}.$$
\noindent Then apply the division algorithm to each term $\displaystyle \frac{a_i}{p_i}$ to get non-negative integers $k_i$ and $x_i$ such that
$$a_i=k_ip_i+x_i, \; \mathrm{with} \; 0\leq x_i <p_i.$$ 
\noindent Let 
$$ k:= \sum _{i=1}^l k_i.$$
\noindent Then $n$ is in the form (\ref{eq:normalform}).

\vspace{0.2cm}

Next we show that the form (\ref{eq:normalform}) is unique.  Suppose 
\begin{equation}\label{eq:normuniq}
p_1\cdots p_l\left ( k_1+\sum _{i=1}^l\frac{x_i}{p_i} \right )=p_1\cdots p_l\left ( k_2+\sum _{i=1}^l\frac{y_i}{p_i} \right ).
\end{equation}
\noindent Taking reductions of both sides modulo $p_i$, we see that 
$$x_i \equiv y_i \pmod  {p_i} \;\;\;\;\;\; \mathrm{for}\; \mathrm{all} \; i=1,\dots, l.$$ 
\noindent Since both $x_i$ and $y_i$ are in $[0,p_i)$, this forces $x_i=y_i$ for all $i=1 ,\dots,  l$. After this observation we see that Equation (\ref{eq:normuniq}) forces $k_1=k_2$.
\end{proof}

\begin{definition}
We shall call the unique expression (\ref{eq:normalform}) in Lemma \ref{lem:normalform}, \emph{the normal form} of the element $n$ of the numerical semigroup $G$.
\end{definition}

\begin{proof}[Proof of Theorem \ref{thm:monotonicity}] 
Let $Y = \Sigma(p_1,\ldots,p_l)$ and $Y' = \Sigma(q_1,\ldots,q_l)$ with $p_i \leq q_i$ for all $i$, and consider the associated abstract delta sequences $(X_Y,\Delta_Y)$ and $(X_{Y'},\Delta_{Y'})$ as described in Section \ref{ss:deltaforseif}. By Corollary \ref{cor:rankimmers}, it suffices to construct an immersion of $\Delta_Y$ into $\Delta_{Y'}$. Let $n\in G_Y$ and write it in the normal form 
\begin{equation}\label{eq:normal}
n=p_1\cdots p_l\left ( k+\sum _{i=1}^l\frac{x_i}{p_i} \right ),
\end{equation}
\noindent for some integers $k\geq 0$ and $0\leq x_i <p_i$, for all $i=1,\dots, l$. Then 
$$n':=q_1\cdots q_l\left ( k+\sum _{i=1}^l\frac{x_i}{q_i} \right )
$$
\noindent is the normal form of an element $n'$ of $G_{Y'}$, since $p_i \leq q_i$ for all $i$. Define $\phi:G_Y\to G_{Y'}$ by 
$$ \phi \left ( p_1\cdots p_l\left ( k+\sum _{i=1}^l\frac{x_i}{p_i} \right  ) \right ) =q_1\cdots q_l\left ( k+\sum _{i=1}^l\frac{x_i}{q_i} \right ).$$
\noindent We claim that $\phi(S_Y)\subseteq S_{Y'}$. To see this, suppose $n\leq N_Y$ and $n\in G_Y$ with normal form (\ref{eq:normal}). Note that 
$ \displaystyle
k+\sum _{i=1}^l\frac{x_i}{p_i} \leq  l-2 -\sum _{i=1}^l\frac{1}{p_i}
$
. Then

$$  \frac{\phi(n)}{q_1\cdots q_l}=k+\sum _{i=1}^l\frac{x_i}{q_i}  \leq  k+\sum _{i=1}^l\frac{x_i}{p_i} \leq  l-2 -\sum _{i=1}^l\frac{1}{p_i} \leq  l-2 -\sum _{i=1}^l\frac{1}{q_i}=\frac{N_{Y'}}{q_1\cdots q_l},$$
\noindent which implies $\phi (n) \leq N_{Y'}$. Hence we have a map $\phi: S_Y \to S_{Y'}$. Next we extend this to a morphism $\phi :(X_Y,\Delta_Y)\to (X_{Y'},\Delta_{Y'})$ by requiring that $\phi(N_Y-x)=N_{Y'}-\phi (x)$, for all $x \in S_Y$.

\vspace{0.2cm}

We claim that $\phi$ is the required immersion. Let us first show that  $\phi$ satisfies
\begin{equation}\label{eq:quas}
\phi(x)+\phi(\tilde{x}) \leq \phi (x+\tilde{x}) \;\;\;\; \mathrm{for}\;\mathrm{all}\;x,\tilde{x}\in G_Y.
\end{equation}
\noindent Write $x$ and $\tilde{x}$ in normal form 
$$x=p_1\cdots p_l\left ( k+\sum _{i=1}^l\frac{x_i}{p_i} \right ),\;\;\;\;\tilde{x}=p_1\cdots p_l\left ( k+\sum _{i=1}^l\frac{\tilde{x}_i}{p_i} \right ).$$
\noindent Observe that both $\phi(x)+\phi(\tilde{x})$ and $\phi (x+\tilde{x})$ belong to $G_{Y'}$. A comparison of their normal forms shows the following:

\begin{itemize}
\item $\phi (x)+\phi (\tilde{x})= \phi (x+\tilde{x})$ if $x_i+\tilde{x}_i <p_i$ or $x_i+\tilde{x}_i\geq q_i$ for all $i=1,\dots, l$.

\item $\phi (x)+\phi (\tilde{x})< \phi (x+\tilde{x})$ if for some $i=1\dots l$, we have $p_i\leq x_i+\tilde{x}_i <q_i$. 
\end{itemize} 

Having shown that $\phi$ satisfies Inequality (\ref{eq:quas}), we will easily prove that $\phi$ is a semi-immersion. Suppose $x \leq y$  for some $x\in S_Y$ and $y\in Q_Y$. Write $y=N_Y -\tilde{x}$ for $\tilde{x}\in S_Y$, then 
$$\phi (x) + \phi (\tilde{x}) \leq \phi (x+\tilde{x}) \leq N_{Y'},$$
\noindent which implies
$$\phi(x) \leq N_{Y'} -\phi (\tilde{x})= \phi (N_Y-\tilde{x})=\phi(y).$$

Finally we verify that $\phi$ is an immersion. By the uniqueness of normal forms, $\phi$ is one-to-one. By part \eqref{i:theo6} of Theorem \ref{theo:mainmore}, $\Delta_Y(x)=\Delta_{Y'}(\phi(x))$, for all $x \in S_Y$. By symmetry  (part \eqref{i:theo3} of Theorem \ref{theo:mainmore}), we conclude $\Delta_Y(y)=\Delta_{Y'}(\phi(y))$, for all $y \in Q_Y$ as well. Hence $\phi$ is an immersion.
\end{proof}

\section{Proof of the Rank Inequality for Pinches}\label{sec:pin}
\subsection{Description}
The goal of this section is to prove Theorem \ref{thm:pinch}. It is clear that in order to prove Theorem~\ref{thm:pinch}, it suffices to prove $\rank \HFred(\Sigma(p_1,\ldots,p_l,qr)) \leq  \rank \HFred(\Sigma(p_1,\ldots,p_l,q,r))$.  We will therefore focus on this inequality for the rest of the section.  In order to obtain the latter, by Corollary~\ref{cor:ranksemi} it suffices to construct a well-behaved one-to-one semi-immersion from $\Delta_Y$ to $\Delta_{Y'}$, where $Y = \Sigma(p_1,\ldots,p_l,qr)$ and $Y' = \Sigma(p_1,\ldots,p_l,q,r)$.

\subsection{Numerical semigroups on two elements}
We are going to define two auxiliary (infinite) delta sequences and a well-behaved semi-immersion between them. This data will be used later in the next subsection in order to define the required well-behaved one-to-one semi-immersion of $\Delta_Y$ into $\Delta_{Y'}$.   Fix relatively prime integers $q,r \geq 2$ throughout this section.  We will study the semigroup $S_{q,r} := \{aq + br | a,b \geq 0\}$.  Recall that the  Frobenius number of $S_{q,r}$ is $qr-q-r$ \cite[Theorem 1.2]{BR}. In other words $qr-q-r\not \in  S_{q,r}$, and every integer $n\geq (q-1)(r-1)$ is in the semigroup $S_{q,r}$. We point out that since $q$ and $r$ are relatively prime, $\frac{(q-1)(r-1)}{2}$ is an integer.

\begin{lemma}\label{lem:semigroupcount} We have
$|\mathbb{N} \smallsetminus S_{q,r}| = |[0,qr-q-r] \cap S_{q,r} | = \frac{(q-1)(r-1)}{2}$.   
\end{lemma}
\begin{proof}
This follows from Sylvester's Theorem which says that precisely half of the integers between $1$ and $(q-1)(r-1)$ belong to  $S_{q,r}$. See for example \cite[Therem 1.3]{BR}   for the proof.
\end{proof}

\begin{lemma}\label{lemm:symm}
Suppose $x\in[0,qr-q-r]$. Then $x\in S_{q,r}$ if and only if $qr-q-r-x\notin S_{q,r}$
\end{lemma}

\begin{proof}
We already know $qr-q-r \notin S_{q,r}$, so  $qr - q - r - x$ is not  in the semigroup if $x\in S_{q,r}$.   Lemma~\ref{lem:semigroupcount} proves that the converse is also true. 
\end{proof}

\begin{lemma}\label{lem:psi}
There exists a unique bijective function $\psi : \mathbb{N} \to S_{q,r}$ such that: 
\begin{enumerate}
\item \label{lempsipart1} for $x_1,x_2 \in \mathbb{N}$, $x_1 < x_2$ implies $\psi(x_1) < \psi(x_2)$, 
\item \label{lempsipart2} if $x \geq \frac{(q-1)(r-1)}{2}$, then $\psi(x) = x + \frac{(q-1)(r-1)}{2}$.  
\end{enumerate}
\end{lemma}
\begin{proof}
First, define $\psi(0) = 0$.  Next, we define $\psi(i)$ to be the $i$th non-zero element in $S_{q,r}$, where the order structure on $S_{q,r}$ is just the induced one from $\mathbb{N}$.  This is clearly the unique order preserving bijection.  We now just want to show that part \eqref{lempsipart2} holds.  As discussed, for all $z > qr - q -r$, we have $z \in S_{q,r}$.  Therefore, if $\psi(x) = z$, where $z > qr - q -r$, then $\psi(x + k) = \psi(x) + k$ for all $k \geq 0$.  It thus suffices to show that $\psi(\frac{(q-1)(r-1)}{2}) = (q-1)(r-1)$.  This follows since there are exactly $\frac{(q-1)(r-1)}{2}$ elements of $[0,(q-1)(r-1)]$ which are not in $S_{q,r}$ by Lemma~\ref{lem:semigroupcount}.  
\end{proof}

\vspace{0.1cm}

 We can define delta functions, $\Delta^{qr}$ on $\mathbb{N}$ by $\Delta^{qr}(x) = 1 + \lfloor \frac{x}{qr} \rfloor$ and $\Delta_{q,r}$ on $S_{q,r}$ by $\Delta_{q,r}(aq + br) = 1 + \lfloor \frac{a}{r} \rfloor + \lfloor \frac{b}{q} \rfloor$.  We regard $(\mathbb{N},\Delta^{qr})$ and $(S_{q,r},\Delta_{q,r})$ as abstract delta sequences, even though both sets are infinite and the delta functions attain only positive values. The map $\psi$ is  a semi-immersion which is also one-to-one.  Therefore, we are still able to study good and bad points as defined in Section \ref{subsec:wel}. As before, $G(\psi)$ and $B(\psi)$ denote the sets of good and bad points of $\psi$ respectively. For any point $x$ belonging to either one of these sets, $d_\psi(x)$ denotes its defect.

\begin{lemma}\label{lem:goodbad}

The one-to-one semi-immersion $\psi:(\mathbb{N},\Delta^{qr})\to (S_{q,r},\Delta_{q,r})$ is well-behaved. In fact $\psi$ admits a control function $\theta_{q,r}:B(\psi)\to G(\psi)$ such that $d_\psi(b)=-d_\psi(\theta_{q,r}(b))=1$ for all $b\in B(\psi)$.
\end{lemma}
\begin{proof}
Before constructing the control function we need to make a few preparations.  First observe the following  about the values of our delta functions. If $k$ is a non-negative integer, and $x$ is an integer in the interval $[kqr,(k+1)qr)$, then $\Delta^{qr}(x) = k+1$. If $x$ is also in $S_{q,r}$, then $\Delta_{q,r}(x) = k$ or $k+1$.  Moreover, $\Delta_{q,r}(x)=k+1$ if and only if $x-kqr \in S_{q,r}$.  

\vspace{0.2cm}

Next we show that $\psi$ shifts all its bad points by $\frac{(q-1)(r-1)}{2}$. To see this, suppose $b\in B(\psi)$. Then by definition $\Delta^{qr}(b)>\Delta_{q,r}(\psi(b))\geq 1$. So $\Delta^{qr}(b)$ is at least 2 which implies $b \geq qr$. By Lemma \ref{lem:psi} part \eqref{lempsipart2}, we have  $\psi(b) = b + \frac{(q-1)(r-1)}{2}$. 

\begin{claim}\label{cl:sm1} 
We have $[kqr, (k+1)qr)\cap B(\psi) \subseteq [kqr,kqr + \frac{(q-1)(r-1)}{2})$, for every integer $k\geq 1$.
\end{claim}
\begin{proof}[Proof of  Claim \ref{cl:sm1}]
Suppose to the contrary that  $b \geq kqr + \frac{(q-1)(r-1)}{2}$ and that $b$ is a bad point of $\psi$. Then $\psi(b) \geq kqr + (q-1)(r-1)$ by the discussion above.  Since $\psi(b) - kqr \geq (q-1)(r-1)$, we have  $\psi(b) - kqr \in S_{q,r}$.  This implies that $\Delta_{q,r}(\psi(b)) \geq k+1$.  Since $b$ is bad, we must have that $\Delta^{qr}(b) \geq k+2$, implying $b\geq (k+1)qr$.  This contradicts $b$ being in $[kqr, (k+1)qr)$.  
\end{proof}

We saw that every bad point $b$  of $\psi$ must be in  $[kqr,kqr + \frac{(q-1)(r-1)}{2})$ for some $k\geq 1$.  We therefore have $\psi (b)\in [kqr,(k+1)qr)$. These give $\Delta^{qr}(b) = k+1$ and $\Delta_{q,r}(\psi(b)) = k$, so the defect of $b$ is $1$.

\vspace{0.2cm}

We are now ready to construct our control function.  Fix $k\geq 1$, and consider $b \in [kqr,kqr + \frac{(q-1)(r-1)}{2})\cap B(\psi)$.    Define
 \begin{equation}\label{eqn:thetaqr}
\theta_{q,r}(b) := 2kqr - b - 1.
\end{equation}
This defines a map $\theta_{q,r}: B(\psi) \to \mathbb{N}$. We will check that $\theta_{q,r}$ satisfies all the properties of being a control function given in Definition~\ref{d:control}. We can see that $\theta_{q,r}(b) \in [kqr - \frac{(q-1)(r-1)}{2}, kqr)$, so $\Delta^{qr}(\theta_{q,r}(b)) = k$ and $\theta_{q,r}(b) < b$.  It is straightforward to check that $\theta_{q,r}$ is injective.  The following claim completes the proof, showing that $\theta_{q,r}(b)$ is good with the appropriate defect.  

\begin{claim}\label{cl:sm2} We have $\Delta_{q,r}(\psi(\theta_{q,r}(b))) = k+1$. Hence the defect of $\theta_{q,r}(b)$ is $-1$.  
\end{claim}
\begin{proof}[Proof of Claim \ref{cl:sm2}]
Since $\Delta_{q,r}(\psi(b)) = k$, we must have 
\[
b + \frac{(q-1)(r-1)}{2} - kqr \not \in S_{q,r}.
\]
Thus $b + \frac{(q-1)(r-1)}{2} - kqr < (q-1)(r-1)$, and Lemma \ref{lemm:symm} implies 
\[
\alpha_b = qr - q - r - \left ( b + \frac{(q-1)(r-1)}{2} - kqr \right ) \in S_{q,r}.
\]
Note that by construction, $\alpha_b \in [0,\frac{(q-1)(r-1)}{2})$.  Let's consider $\beta_b$ given by 
\[
\beta_b = kqr + \alpha_b = (2k+1)qr - q - r - b - \frac{(q-1)(r-1)}{2}.  
\]    
We have that $\beta_b \in [kqr, kqr + \frac{(q-1)(r-1)}{2})$.  
Since $\alpha_b \in S_{q,r}$ and $\beta_b \in [kqr, kqr + \frac{(q-1)(r-1)}{2})$, we have $\Delta_{q,r}(\beta_b) = k+1$.   
Furthermore, since $\theta_{q,r}(b) \geq \frac{(q-1)(r-1)}{2}$, Lemma \ref{lem:psi} gives
\[
\psi(\theta_{q,r}(b)) = \theta_{q,r}(b) + \frac{(q-1)(r-1)}{2} = 2kqr - b - 1 + \frac{(q-1)(r-1)}{2} = \beta_b,
\]  
and thus 
\[
\Delta_{q,r}(\psi(\theta_{q,r}(b))) = \Delta_{q,r}(\beta_b) = k+1.
\]    
\end{proof}
\end{proof}

\subsection{Constructing the map}
We complete the proof of Theorem~\ref{thm:pinch} in two steps.  The first is to find a rigid function (see Definition \ref{def:tight}) from $G_Y$ to $G_{Y'}$ where  $Y = \Sigma(p_1,\ldots,p_l,qr)$ and $Y' = \Sigma(p_1,\ldots,p_l,q,r)$, and $G_Y$ and $G_{Y'}$ are their associated semigroups given in Theorem \ref{theo:mainmore}.  The second is to show that the semi-immersion induced by Lemma~\ref{lem:tightimpliesrightveering} is well-behaved.  We would like a standard form for studying the elements of $G_Y$ and $G_{Y'}$.  We will always assume that an element $x$ of $G_Y$ is expressed in the form
\[
x =p_1 \cdots p_l \left [ z  + qr\left ( \sum^l_{i=1}  \frac{x_i}{p_i} \right ) \right ], 
\]
where $0 \leq x_i < p_i$, but we may have $z$ be arbitrarily large. Note that this is different from the normal form introduced in Section~\ref{sec:par}.  Like the normal form, this expression is unique.  We define a function $\pi : G_Y \to \mathbb{N}$ by $\pi(x) = z$.  

\vspace{0.2cm}

Similarly for $G_{Y'}$, we have a  decomposition
\[
x' = p_1 \cdots p_l \left [ aq  + br  + qr\left (\sum^l_{i=1}  \frac{x'_i}{p_i} \right ) \right ],
\]         
where $0 \leq x'_i < p_i$, but $a$ and $b$ may be arbitrarily large.  We define $\pi':G_{Y'} \to S_{q,r}$ by $\pi'(x') = aq + br$. Though the above decomposition is not unique,  the map $\pi '$ is well-defined.  

\vspace{0.2cm}

The values of $\Delta_Y$ and $\Delta_{Y'}$ can be  computed using Theorem \ref{theo:mainmore}:
\begin{equation}
\Delta_Y(x) = \Delta^{qr}(\pi(x)) \text{ and } \Delta_{Y'}(x') = \Delta_{q,r}(\pi'(x')).   
\end{equation}

Now define $\phi: G_Y \to G_{Y'}$ by 
\begin{equation}\label{eqn:phi}
\phi \left ( p_1 \cdots p_l \left [ z  + qr\left ( \sum^l_{i=1}  \frac{x_i}{p_i} \right ) \right ]\right ) = p_1 \cdots p_l \left [ \psi (z)  + qr\left ( \sum^l_{i=1}  \frac{x_i}{p_i} \right ) \right ],
\end{equation}
where $\psi$ is defined as in Lemma~\ref{lem:psi}.  Note that since $\psi(z) \in S_{q,r}$, the codomain of $\phi$ is  $G_{Y'}$.  

\begin{lemma}
The function $\phi$ is rigid. 
\end{lemma}
\begin{proof}
Recall from Lemma~\ref{lem:psi} that $\psi$ is injective. This implies $\phi$ is also injective. Since $\psi(\pi(x)) \geq \pi(x)$ for all $x \in \mathbb{N}$, we have  $\phi(x) \geq x$ for all $x \in G_Y$.  
We also have that 
\begin{align*}
N_{Y'} - N_Y & = p_1 \cdots p_l\left ( \left [ lq r - (q+r) - qr \left (\sum^l_{i=1} \frac{1}{p_i} \right )\right ]   - \left [ (l-1)qr -1 - qr \left (\sum^l_{i=1}  \frac{1}{p_i}\right ) \right ]\right )\\
& = p_1 \cdots p_l(qr - q - r + 1) = p_1 \cdots p_l(q - 1)(r-1).  
\end{align*}
Therefore by Lemma~\ref{lem:psi}, 
\begin{align*}
\phi(x) - x & = p_1 \cdots p_l (\psi(\pi(x)) - \pi(x)) \\
& \leq p_1 \cdots p_l \frac{(q-1)(r-1)}{2}  \\
&= \frac{N_{Y'} - N_Y}{2}.  
\end{align*}
Thus, $\phi$ is rigid.  
\end{proof}

\begin{proof}[Proof of Theorem~\ref{thm:pinch}]
As discussed, it suffices to show 
\[
\rank \HFred(Y) \leq \rank \HFred(Y')
\]
where $Y = \Sigma(p_1,\ldots,p_l,qr)$ and $Y' = \Sigma(p_1,\ldots,p_l,q,r)$. We have given a rigid map $\phi: G_Y \to G_{Y'}$ in Equation (\ref{eqn:phi}).  We study the induced one-to-one semi-immersion $\phi: (X_Y,\Delta_Y) \to (X_{Y'},\Delta_{Y'})$ (see Lemma \ref{lem:tightimpliesrightveering}).  We want to see that $\phi$ is well-behaved.  Hence we need to construct a control function $\theta_Y:B(\phi)\to G(\phi)$ in the sense of Definition \ref{d:control}. First we need to do a few preparations. Throughout we shall assume that every element $x\in S_Y$ is written in the form
\begin{equation}\label{eq:stand}
x=p_1\cdots p_l\left [\pi (x) + qr  \left (\sum _{i=1}^l \frac{x_i}{p_i}\right ) \right ] .
\end{equation}

\begin{claim}\label{cl:bg1}
Suppose $x\in S_Y$. Then $d_\phi (x) =d_\psi (\pi (x))$. In particular, $x\in B(\phi)$ if and only if $\pi (x) \in B(\psi)$. Similarly, $x\in G(\phi)$ if and only if $\pi (x) \in G(\psi)$. 
\end{claim}

\begin{proof}[Proof of Claim \ref{cl:bg1}]
We simply calculate the defect of $x$ using Equation~\eqref{eq:stand} to see that
\[
d_\phi(x)=|\Delta_Y(x)| - |\Delta_{Y'}(\phi(x))| = \Delta^{qr}(\pi(x)) - \Delta_{q,r}(\psi(\pi(x)))=d_\psi(\pi(x)).  
\]
\end{proof}
\begin{claim}\label{cl:bg2}
Suppose $y\in Q_Y$ with $y= N_Y-x$ where $x\in S_Y$. Then $d_\phi (y) =d_\psi (\pi (x))$. In particular, $y\in B(\phi)$ if and only if $\pi (x) \in B(\psi)$. Similarly, $y\in G(\phi)$ if and only if $\pi (x) \in G(\psi)$. 
\end{claim}

\begin{proof}[Proof of Claim \ref{cl:bg2}]
Again we calculate the defect of $y$ using Equation~\eqref{eq:stand} to see that
\[
d_\phi(y)=|\Delta_Y(y)| - |\Delta_{Y'}(\phi(y))| = \Delta^{qr}(\pi(x)) - \Delta_{q,r}(\psi(\pi(x)))=d_\psi(\pi(x)).  
\]
\end{proof}

We are now ready to define the control function $\theta_Y$. First define it on $S_Y\cap B(\phi)$.  Suppose $x\in S_Y \cap B(\phi)$. Then we let
\[
\theta_Y(x):= p_1 \cdots p_l \left [ \theta_{q,r}(\pi (x)) + qr\left (\sum^l_{i=1}  \frac{x_i}{p_i} \right ) \right ],
\] 
where $\theta_{q,r}$ is as defined in Equation (\ref{eqn:thetaqr}). Note that Claim \ref{cl:bg1} implies that the term $\theta_{q,r}(\pi (x))$ above makes sense.  Next define $\theta_Y$ on $Q_Y\cap B(\phi)$.  Suppose $y\in Q_Y\cap B(\phi)$ with $y=N_Y-x$. Then we let
\[
\theta_Y(y):= N_Y- p_1 \cdots p_l \left [ \theta_{q,r}(\pi (x)) + qr\left (\sum^l_{i=1}  \frac{x_i}{p_i} \right ) \right ].
\] 
Note that Claim \ref{cl:bg2} implies that the term $\theta_{q,r}(\pi (x))$ above makes sense.

\vspace{0.2cm}

Having seen that  the map $\theta_Y$ is defined on  $B(\phi)$, let us now show that it has the correct codomain. We must show that $ 0 \leq \theta_Y(b) \leq N_Y$ for all $b \in B(\phi)$. Indeed, since $\theta_{q,r}$ is a control function, $0 \leq \theta_{q,r}(\pi (x))<\pi (x)$ for all $x\in S_Y\cap B(\phi)$.  This implies
\begin{align}
0 & \leq \theta_Y(x) < x < N_Y \; \mathrm{for}\; \mathrm{all} \; x \in S_Y\cap B(\phi),\label{est:a}\\
0 & < y < \theta_Y(y) \leq N_Y \; \mathrm{for}\; \mathrm{all} \; y \in Q_Y\cap B(\phi)\label{est:b}.
\end{align}
\noindent It is now clear that we have a map $\theta_Y:B(\phi)\to X_Y$ such that $ \theta_Y(S_Y\cap B(\phi)) \subseteq S_Y$ and $\theta_Y(Q_Y\cap B(\phi)) \subseteq Q_Y$. 

\vspace{0.2cm}

We start checking that $\theta_Y$ satisfies the properties listed in Definition \ref{d:control}. Since $\theta_{q,r}$ is a control function, Claim \ref{cl:bg1} and Claim \ref{cl:bg2} respectively imply that
\begin{align*}
\theta_Y(S_Y \cap B(\phi)) & \subseteq S_Y \cap G(\phi),\\
\theta_Y(Q_Y \cap B(\phi)) & \subseteq Q_Y \cap G(\phi).
\end{align*} 
\noindent Hence $\theta:B(\phi)\to G(\phi)$ satisfies Property (\ref{def:cont1}) of Definition~\ref{d:control}.  Clearly $\theta_Y$ is injective since $\theta_{q,r}$ is injective. Property (\ref{def:cont2}) was shown in Inequalities (\ref{est:a}) and (\ref{est:b}) above.  Finally Lemma~\ref{lem:goodbad} together with Claim \ref{cl:bg1} and Claim \ref{cl:bg2} imply that $|d_\phi(b)|=|d_\phi(\theta_Y(b))|=1$ for all $b\in B(\phi)$. Hence Property (\ref{def:cont3}) follows.

\vspace{0.2cm}

Therefore, we have shown that the one-to-one semi-immersion $\phi$ admits a control function, and hence it is well-behaved.  The theorem now follows from Corollary~\ref{cor:ranksemi}.

\end{proof}

\section{Botany}\label{s:botany}

In this section we will prove Theorem~\ref{thm:botany} and some other results related to the botany question. 

\begin{proof}[Proof of Theorem \ref{thm:botany}]
Without loss of generality assume the multiplicities of the singular fibers are in the increasing order,
$$1<p_1<p_2<\ldots<p_l.$$

\noindent Therefore $p_j> j$ for all $j=1,\dots,l$. Now 	$n=\rank \HFred (\Sigma (p_1,\ldots,p_l)) \geq 1$ rules out the possibility of $l\leq 2$, since $S^3$ is the only Seifert homology sphere with at most $2$ singular fibers and $\HFred (S^3)=0$. Part (\ref{i:bot1}) is trivial when $l=3$. Suppose  $l\geq 4$. We have $(p_1,p_2,p_3,p_4)\geq (2,3,5,7)$. A direct computation shows $\HFred (\Sigma(2,3,5,7))=13$. By repeated applications of Theorem~\ref{thm:pinch} and Theorem~\ref{thm:fiberbranched}, we have 
\begin{align*}
n = \rank \HFred (\Sigma(p_1,\ldots,p_l)) &\geq \rank \HFred (\Sigma(p_1,\ldots,p_{l-1}p_l))\\
&\geq p_l \rank \HFred (\Sigma(p_1,\ldots,p_{l-1}))\\
&\vdots \\
&\geq p_l \cdots p_5 \rank \HFred (\Sigma(p_1,p_2,p_3,p_4))\\
&\geq \frac{l!}{4!}\rank \HFred (\Sigma(2,3,5,7))\quad \text{(by Theorem \ref{thm:monotonicity})} \\
&> \frac{l!}{2},
\end{align*}
which establishes part (\ref{i:bot1}) of the theorem.

\vspace{0.2cm}

For part (\ref{i:bot2}), suppose to the contrary that $p_l\geq 6n+7$. One can directly verify that the rank of $\HFred (\Sigma(2,3,6n+7))$ is $n+1$. By repeatedly applying Theorem~\ref{thm:pinch} and Theorem~\ref{thm:fiberbranched}, we get
\begin{align*}
n =\rank \HFred (\Sigma(p_1,\ldots,p_l)) &\geq p_1\ldots p_{l-3}  \rank \HFred (\Sigma(p_{l-2},p_{l-1},p_l))\\
& \geq  \rank \HFred (\Sigma(2,3,6n+7)) \quad \text{(by Theorem \ref{thm:monotonicity})}\\
&=n+1,
\end{align*}

\noindent which is a contradiction.
\end{proof} 

 Recall that for a  Spin$^c$ structure $\spinc$ on a rational homology sphere $Y$, the group $HF^+(Y,\spinc)$ comes equipped with a $\mathbb{Q}$-valued grading \cite{OS2}.  The correction term, $d(Y,\spinc)$, is defined to be the minimal grading of an element of $HF^+(Y,\spinc)$ which lies in the image of the obvious map from $HF^\infty(Y,\spinc)$ to $HF^+(Y,\spinc)$.  In the case of a homology sphere, we omit the Spin$^c$ structure from the notation.  Furthermore, in this case, the correction term is always an even integer.

\begin{corollary}\label{cor:finiteness}
For fixed $n \in \mathbb{N}$, there are at most finitely many Seifert homology spheres $Y$ with $\rank \HFred (Y) \leq n$.  Therefore, there are at most finitely many integers $d$ that can occur as the correction term of a Seifert homology sphere $Y$ with $\rank \HFred (Y)= n$.  
\end{corollary}
\begin{proof}
By Theorem~\ref{thm:botany}, $n$ easily provides an upper bound on the number of singular fibers, and their multiplicities for which $\rank \HFred = n$.  
\end{proof}

\begin{corollary}\label{cor:cassonfiniteness}
For a fixed integer $n$, there are at most finitely many Seifert homology spheres $Y$ with $\lambda(Y) = n$, where $\lambda$ denotes the Casson invariant.  
\end{corollary}
\begin{proof}
We apply \cite[Theorem 1.3]{OS2} relating $\HFred$, the correction term, and the Casson invariant: 
\[
\lambda(Y) + \frac{d(Y)}{2} = \chi(\HFred(Y)).
\]
After a possible change in orientation, $Y$ bounds a positive-definite four-manifold, arising from a plumbing.  It is shown in \cite{OS1} that $HF^+(Y)$ is supported in even gradings.  Therefore, $\chi(\HFred(Y)) = \rank \HFred(Y)$.  Furthermore, in this case, $d(Y) \leq 0$.  Therefore, $\lambda(Y) \geq \rank \HFred(Y)$.  Corollary~\ref{cor:finiteness} now gives the proof.  
\end{proof}

\subsection{Botany of Seifert homology spheres with rank at most $12$} Here we list all the Seifert homology spheres whose rank of the reduced Heegaard Floer homology is at most 12. By Theorem~\ref{thm:botany}, these Seifert homology spheres necessarily have $3$ singular fibers excepting the $3$--sphere. The same result says that the maximum multiplicity of these singular fibers is no more than $79$. We looked at all the triples  of pairwise relatively prime integers $(p_1,p_2,p_3)$ with $1<p_1<p_2<p_3<79$ and calculated the rank of  reduced Heegaard Floer homology of the corresponding Seifert homology spheres. The triples listed in Table~\ref{tab:rank} are the only ones with rank no more than $12$.

\tiny

\begin{table}[h]
\begin{tabular}{|l|l|}
\hline
Rank & Triple\\
\hline
$n=0$ & $S^3, [2,3,5]$ \\
\hline
$n=1$&  $ [ 2, 3, 11 ], [ 2, 3, 7 ]$\\
\hline
$n=2$& $ [ 3, 4, 5 ], [ 2, 3, 17 ], [ 3, 4, 7 ], [ 2, 3, 13 ], [ 2, 5, 9 ],  [ 2, 5, 7 ]$\\
\hline
$n=3$		 & $  [ 2, 5, 13 ], [ 2, 3, 23 ],  [ 2, 7, 9 ],  [ 2, 5, 11 ], [ 2, 3, 19 ], [ 3, 5, 7 ] $\\
\hline
	
$n=4$&$  [ 3, 4, 11 ],
    [ 2, 7, 13 ],
    [ 3, 5, 8 ],
    [ 2, 3, 25 ],
    [ 2, 7, 11 ],
    [ 2, 3, 29 ]$\\
	
\hline
		
$n=5$&$ [ 4, 5, 7 ],
    [ 3, 4, 13 ],
    [ 3, 5, 11 ],
    [ 2, 3, 35 ],
    [ 2, 5, 17 ],
    [ 2, 3, 31 ],
    [ 2, 5, 19 ]$\\
	
\hline
	
$n=6$&$ [ 2, 7, 15 ],
    [ 2, 5, 21 ],
    [ 2, 5, 23 ],
    [ 4, 5, 9 ],
    [ 3, 7, 8 ],
    [ 3, 5, 13 ],
    [ 2, 9, 11 ],
    [ 3, 5, 14 ],
    [ 2, 3, 41 ],
    [ 2, 3, 37 ]$\\ 
	
\hline
	
$n=7$&$    [ 2, 3, 47 ],
    [ 3, 4, 17 ],
    [ 3, 7, 10 ],
    [ 3, 4, 19 ],
    [ 2, 7, 17 ],
    [ 2, 3, 43 ],
    [ 2, 9, 13 ]$\\ 
	
\hline

$n=8$&$   [ 2, 3, 49 ],
    [ 3, 7, 11 ],
    [ 2, 7, 19 ],
    [ 2, 5, 29 ],
    [ 2, 3, 53 ],
    [ 2, 11, 13 ],
    [ 4, 5, 11 ],
    [ 2, 9, 17 ],
    [ 3, 5, 17 ],
    [ 2, 5, 27 ],
    [ 5, 6, 7 ],
    [ 3, 5, 16 ],
    [ 4, 7, 9 ]$\\ 
	
\hline

$n=9$&$     [ 4, 5, 13 ],
    [ 2, 5, 33 ],
    [ 2, 5, 31 ],
    [ 2, 3, 59 ],
    [ 2, 3, 55 ],
    [ 3, 8, 11 ],
    [ 3, 7, 13 ],
    [ 3, 4, 23 ],
    [ 2, 7, 23 ] $ \\
	
\hline

$n=10$&$     [ 2, 3, 61 ],
    [ 3, 4, 25 ],
    [ 2, 11, 15 ],
    [ 2, 3, 65 ],
    [ 5, 7, 8 ],
    [ 2, 9, 19 ],
    [ 2, 7, 27 ],
    [ 3, 5, 19 ],
    [ 2, 7, 25 ],
    [ 2, 11, 17 ] $ \\
	
\hline
$n=11$&$       [ 2, 3, 71 ],
    [ 3, 5, 22 ],
    [ 2, 3, 67 ],
    [ 2, 5, 37 ],
    [ 3, 8, 13 ],
    [ 2, 5, 39 ],
    [ 4, 7, 11 ] $ \\
	
\hline

$n=12$&$    [ 5, 7, 9 ],
    [ 4, 5, 19 ],
    [ 4, 5, 17 ],
    [ 2, 11, 19 ],
    [ 2, 11, 21 ],
    [ 3, 7, 17 ],
    [ 3, 7, 16 ],
    [ 2, 5, 41 ],
    [ 2, 9, 23 ],$\\
		&  $[ 2, 5, 43 ],
    [ 2, 7, 29 ],
    [ 2, 3, 73 ],
    [ 3, 10, 11 ],
    [ 3, 4, 29 ],
    [ 2, 3, 77 ],
    [ 5, 6, 11 ],
    [ 3, 4, 31 ],
    [ 2, 13, 15 ],
    [ 3, 5, 23 ] $ \\
	
\hline

\end{tabular}
\caption{Seifert homology spheres with rank of reduced Heegaard Floer homology less than or equal to $12$.}
\label{tab:rank}
\end{table}

\normalsize

\section{Maps between Seifert homology spheres}\label{sec:nonzerodegree}
It is our goal in this section to prove Theorem~\ref{thm:nonzeroexist} and therefore we must study maps $f:Y' \to Y$ between Seifert homology spheres. First of all, it is clear that Theorem~\ref{thm:nonzeroexist} trivially holds if the degree of $f$ is zero.  Thus, for the rest of this section, we only study non-zero degree maps. Furthermore we always choose orientations of $Y$ and $Y'$ such that $\deg f>0$.  We will determine which pairs of Seifert  homology spheres admit non-zero degree maps between them and what their possible degrees are.  We will then use this in conjunction with the inequalities in Theorems~\ref{thm:fiberbranched} and \ref{thm:pinch} to prove Theorem~\ref{thm:nonzeroexist}.  It turns out that the problem of the existence of non-zero degree maps between Seifert fibered spaces with infinite fundamental groups is settled \cite{Hu, Ro2, Ro1}, and we will specialize this to the case of homology spheres.  

\vspace{0.2cm}
 Before beginning our discussion we set up some conventions which will be used throughout this section.   We shall introduce and use unnormalized Seifert invariants (see for example \cite[Sections 2.14 and 2.15]{NN}).  Let $Y$ be a Seifert homology sphere. Let $(e_0,(p_1',p_1),\dots,(p'_t,p_t))$ be its normalized Seifert invariants. We define the \emph{Euler number} of $Y$ to be
$$e(Y):=e_0+\sum_{i=1}^t \frac{p'_i}{p_i}.$$
The \emph{unnormalized Seifert invariants} of $Y$ is a set $\displaystyle \left \{\frac{\omega_1}{\alpha_1},\dots,\frac{\omega_k}{\alpha_k} \right \}$ with $k\geq t$ satisfying

\begin{enumerate}
\item $\alpha_i=p_i$ for all $i=1,\dots,t$,
\item  $\omega_i\equiv-p'_i \mod  p_i$ for all $i=1,\dots,t$,
\item $\gcd(\omega_i,\alpha_i)=1$ for all $i=1,\dots,t$,
\item $\displaystyle e(Y)=-\sum_{i=1}^k \frac{\omega_i}{\alpha_i}$.
\end{enumerate}
Note that the unnormalized Seifert invariants are not unique, but they uniquely determine the normalized Seifert invariants and hence the Seifert homology sphere $Y$.  When referring to a Seifert homology sphere $\Sigma(p_1,\ldots,p_t)$, we will also allow $p_i = 1$, unless stated otherwise.

\vspace{0.2cm}
		
We begin with the types of maps that can exist between Seifert fibered spaces.

\begin{theorem}[Rong, Theorem 3.2 of \cite{Ro2}]\label{thm:rongnonzero}
Suppose that $f:Y' \to Y$ is a map between $P^2$-irreducible Seifert fibered spaces with infinite fundamental group.  Then, $f$ is homotopic to a map $p \circ g \circ \pi$, where $\pi$ is a degree one map between Seifert fibered spaces, $g$ is a fiber-preserving branched cover branched along fibers, and $p$ is a covering.  Furthermore, if $Y$ is not a Euclidean manifold, we can choose $p = id$.  
\end{theorem}

Rather than define $P^2$-irreducible manifolds, we point out that Seifert homology spheres are always $P^2$-irreducible.  Theorem~\ref{thm:rongnonzero} is actually more specific, stating that the map $\pi$ is a special type of degree one map, called a {\em vertical pinch} which is defined below.  Since Seifert homology spheres never have Euclidean geometry, we may further assume $f$ is homotopic to $g \circ \pi$ as above.  

\vspace{0.2cm}

The following is a standard fact about non-zero degree maps, but we give a proof for completeness.
\begin{proposition}\label{fact:nonzero}
If $f:Y' \to Y$ is a non-zero degree map between closed, connected, orientable three-manifolds, the index of $f_*(\pi_1(Y'))$ as a subgroup of $\pi_1(Y)$ divides the degree of $f$, and thus is finite.  In particular, if $Y'$ is an integer homology sphere and the degree of $f$ is one, then $Y$ is an integer homology sphere as well.
\end{proposition}
\begin{proof}
Let $\widetilde{\pi}:\widetilde{Y}\to Y$ be the covering corresponding to the subgroup $f_*(\pi_1(Y'))$. Then we can lift $f$ to a map $\widetilde{f}:Y'\to \widetilde{Y}$ such that $f=\widetilde{\pi}\circ \widetilde{f}$. If the index of $f_*(\pi_1(Y'))$ is not finite then $\widetilde{Y}$ is not compact hence $H_3(\widetilde{Y})$ is trivial. This implies that $f:H_3(Y') \to H_3(Y)$ factors through the trivial group. Therefore, $\deg f=0$. If the index of $f_*(\pi_1(Y'))$ is finite then it is equal to $\deg \widetilde{\pi}$ and we have $\deg f=\deg \widetilde{\pi} \deg \widetilde{f}$. In particular, the index divides $\deg f$.  In the case that $\deg f =1$, we see that $f_*:\pi_1(Y') \to \pi_1(Y)$ is surjective.  Abelianizing this map shows that $H_1(Y')$ surjects onto $H_1(Y)$ if $\deg f=1$. Hence, $Y$ is an integral homology sphere if $Y'$ is.

\end{proof}

We also note that the only Seifert homology spheres with finite fundamental group are $S^3$ and $\Sigma(2,3,5)$.  Since $\HFred(S^3) = \HFred(\Sigma(2,3,5)) = 0$, Theorem~\ref{thm:nonzeroexist} is immediately satisfied if $\pi_1(Y)$ is finite.  Thus, it suffices to consider the case that $\pi_1(Y)$ is infinite.  By Proposition~\ref{fact:nonzero}, we may assume that $\pi_1(Y')$ is also infinite.  

\vspace{0.2cm}

Therefore, given a non-zero degree map, $f$, between Seifert homology spheres with infinite fundamental group, we may factor $f \simeq g \circ \pi$ as in Theorem~\ref{thm:rongnonzero} where the codomain of $\pi$/domain of $g$ is also a Seifert homology sphere with infinite fundamental group.  In order to prove Theorem~\ref{thm:nonzeroexist}, we therefore analyze the pairs of Seifert homology spheres with infinite fundamental group which admit non-zero degree maps between them via the next two propositions.  We will use the abbreviation ISHS for a Seifert homology sphere with infinite fundamental group. 

\begin{proposition}\label{prop:classificationpinch}
Let $\pi:Y' \to Y$ be a degree one map between ISHS's.  Then $Y$ can be obtained from $Y'$ by a sequence of moves (up to permutation of the multiplicities of the singular fibers) of the form 
\[\Sigma(p_1,\ldots,p_l) \to \Sigma(p_1,\ldots,p_k,p_{k+1} \cdots p_l).\]
\end{proposition}

\begin{proposition}\label{prop:classificationbranched}
If $f:Y' \to Y$ is a fiber-preserving branched cover between ISHS's, then $Y'$ admits a map to $Y$ with the same degree as $f$ which is obtained by a composition (up to permutation of the multiplicities of the singular fibers) of branched covers of the form: 
\begin{itemize}
\item $\phi: \Sigma(p_1,\ldots,n p_l) \to \Sigma(p_1,\ldots,p_l)$ or
\item $\psi: \Sigma(p_1,\ldots,p_l,n) \to \Sigma(p_1,\ldots,p_l)$, 
\end{itemize}
where $\phi$ (respectively $\psi$) is the $n$-fold cyclic branched cover, branched over the singular fiber of order $p_l$ (respectively a regular fiber).  
\end{proposition}

Before giving the proofs, we will see how these two propositions lead to a proof of Theorem~\ref{thm:nonzeroexist}.  
\begin{proof}[Proof of Theorem~\ref{thm:nonzeroexist}]
As discussed, we only need to consider the case of a non-zero degree map $f:Y' \to Y$ between ISHS's.  We therefore factor $f$ as $g \circ \pi$ as in Theorem~\ref{thm:rongnonzero}.  Since degree is multiplicative under composition, we can prove Inequality \eqref{eqn:rankinequality} for $\pi$ and $g$ separately.  The inequality for $\pi$ follows from Proposition~\ref{prop:classificationpinch} and Theorem~\ref{thm:pinch}.  On the other hand, for $g$, we observe that by combining Theorems~\ref{thm:fiberbranched} and ~\ref{thm:pinch}, we obtain 
\[ n \rank \HFred(\Sigma(p_1,\ldots,p_l)) \leq \rank \HFred(\Sigma(p_1,\ldots,np_l)) \leq \rank \HFred(\Sigma(p_1,\ldots,p_l,n)) . \]
The result now follows from Proposition~\ref{prop:classificationbranched}, since an $n$-fold branched covering is a degree $n$ map.  
\end{proof}

The rest of this section is devoted to the proofs of Proposition~\ref{prop:classificationpinch} and Proposition~\ref{prop:classificationbranched}, which will thus complete the proof of Theorem~\ref{thm:nonzeroexist}.

\vspace{0.2cm}

In order to prove Proposition~\ref{prop:classificationpinch}, we recall a special kind of degree one map.  We begin by fixing a Seifert fibered space $M$ with a separating torus $T$ which is vertical (i.e. $T$ is foliated by fibers).  Decompose $M$ along $T$ into two components, $M_1$ and $M_2$.  Furthermore, suppose that there exists an essential simple closed curve on $T$ which bounds a 2-sided surface in $M_2$.  A degree one map $f$ from $M$ to $M_1 \cup_T D^2 \times S^1$ is a {\em vertical pinch} if $f|_{M_1}$ is the identity and $f$ maps $M_2$ onto $D^2 \times S^1$.  It is straightforward to check that there exists a vertical pinch $\Sigma(p_1,\ldots,p_l,q,r) \to \Sigma(p_1,\ldots,p_l,qr)$ (in this case $M_1 = D^2(p_1,\ldots,p_l)$ and $M_2 = D^2(q,r)$).  With this definition, we are now able to describe the degree one maps that appear in Theorem~\ref{thm:rongnonzero}.   

\begin{theorem}[Rong, Corollary 3.3 of \cite{Ro2}]\label{the:rorong}
Let $f:Y' \to Y$ be a degree one map between closed, $P^2$-irreducible Seifert fibered spaces with infinite fundamental group.  Then, $f$ is a composition of vertical pinches.  
\end{theorem}

\begin{proof}[Proof of Proposition~\ref{prop:classificationpinch}]
Let $f:Y' \to Y$ be a degree one map between ISHS's.  We write the unnormalized Seifert invariants of $Y'$ as $\{\frac{q'_1}{q_1},\ldots,\frac{q'_k}{q_k}\}$. By Theorem~\ref{the:rorong}, $f$ is a composition of vertical pinches. We now apply \cite[Theorem 3.2]{Ro1} which states that there are unnormalized Seifert invariants $\{\frac{p'_1}{p_1},\ldots,\frac{p'_l}{p_l}\}$ for $Y$ and a partition $\{1,\ldots,k\} = I_1 \cup \ldots\cup  I_l$ such that for each $j$, 
\begin{equation}\label{eqn:pinchcondition} 
\sum_{n \in I_j} \frac{q'_n}{q_n} = \frac{p'_j}{p_j}, \ \lcm_{n \in I_j}\{q_n\} = p_j.
\end{equation}  
Since $q_i$ (respectively $p_i$) are relatively prime, it is straightforward to check that Condition~\eqref{eqn:pinchcondition} implies that $p_j = \prod_{n \in I_j} q_n$.  Therefore, the relation between $Y'$ and $Y$ is as in the statement of the proposition.       
\end{proof}

Now suppose that $f:Y' \to Y$ is a fiber-preserving branched cover between ISHS's.  Fix unnormalized Seifert invariants for $Y'$ and $Y$ as $\{\frac{q'_1}{q_1},\ldots,\frac{q'_k}{q_k}\}$ and $\{\frac{p'_1}{p_1},\ldots,\frac{p'_l}{p_l}\}$ respectively, arranging that $p_j \geq 2$ and $\frac{p'_j}{p_j} \neq 0$ for each $j$.  
We use $S'$ and $S$ to denote the base orbifolds of $Y'$ and $Y$ respectively. The orders of the orbifold points correspond to the multiplicities of the singular fibers of the Seifert fibered space sitting over it.  Since $f$ is fiber-preserving, we obtain an induced map between the base orbifolds $F: S' \to S$.  We call the degree of $F$ the {\em orbifold degree} of $f$.  Fix a regular fiber, $h'$, of $Y'$ which $f$ maps to a regular fiber, $h$, in $Y$, not contained in the branch set of $f$.  The degree of $f|_{h'}:h' \to h$ is called the {\em fiber degree of $f$}.  It follows that the degree of $f$ is the product of the fiber degree and the orbifold degree.  While there are many fiber-preserving branched covers between Seifert fibered spaces with arbitrary orbifold degree, we will see this does not happen for ISHS's.  
The following argument was shown to us by Ian Agol.  We point out that the ideas are also similar to those used in the proof of \cite[Lemma 2.1]{Hu}.   

\begin{proposition}\label{prop:degreeisfiberdegree}
Let $f:Y' \to Y$ be a fiber-preserving branched cover between ISHS's.  Then, the fiber degree of $f$ equals the degree of $f$.
\end{proposition}
\begin{proof}

We argue by contradiction.  Therefore, we assume that the map $F: S' \to S$ on the underlying topological spaces has degree at least two, and thus is a non-trivial branched cover.  Let $d$ denote this degree.  The first claim is that the orbifold points of $S$ must be contained within the branch set.  Otherwise, there would exist an orbifold point $x \in S$ with order at least two which would lift to more than one orbifold point in $S'$ with the same order.  Therefore, $Y'$ would have two singular fibers with the same multiplicity (of at least two), contradicting $Y'$ being an integer homology sphere.   

\vspace{0.2cm}

We write $x_1,\ldots, x_l$ to denote the orbifold points of $S'$, where we write the order of $x_i$ as $p_i$.  We also let $B_i = f^{-1}(x_i)$ and write $B_i = \{a_{i1},\ldots,a_{i|B_i|}\}$.  Denote by $m_{ij}$ the order of the orbifold point $a_{ij}$ (which may be 1).  In other words, 
\[
Y' = \Sigma(q_1,\ldots,q_k) \cong \Sigma(m_{11},\ldots,m_{1|B_1|},\ldots,m_{l1},\ldots,m_{l|B_l|}).
\]
We would like to see how the $m_{ij}$, $p_i$, and $d$ are related.  Let $d_{ij}$ denote the local degree of $F$ (between the underlying topological spaces) at $a_{ij}$; clearly $\sum_j d_{ij} = d$ for each $i$.  We also have $m_{ij} = \frac{p_i}{\gcd(d_{ij},p_i)}$ for all $i$ and $j$.        

\vspace{0.2cm}
    
Fix some $1 \leq i \leq l$ and let $r_i$ denote the smallest prime which divides $p_i$.  We claim that 
\begin{equation}\label{eqn:multiplicitybounds}
|B_i| \leq \frac{d + r_i - 1}{r_i}.  
\end{equation}
This is seen as follows.  Since $Y'$ is an integer homology sphere, $\gcd(m_{ij},m_{ij'}) = 1$ for all $j, j'$.  In particular, there is at most one $j$ such that $d_{ij} = 1$.  We denote this index by $j_*$, should it exist.  For all $j \neq j_*$, some non-trivial factor of $d_{ij}$ must divide $p_i$.  Therefore, we must have that $d_{ij} \geq r_i$.  Thus, we see 
\[
d = \sum_j d_{ij} = \sum_{j \neq j_*} d_{ij} + d_{ij_*} \geq (|B_i| - 1) r_i + 1, 
\]
which gives Inequality \eqref{eqn:multiplicitybounds}.  We will apply this relation to contradict the Riemann-Hurwitz formula applied to $F$, which states \begin{equation}\label{eqn:riemannhurwitz}
\sum^l_{i = 1} |B_i| \geq (l-2)d + 2,
\end{equation}
because $S$ and $S'$ are spheres.

\vspace{0.2cm}

Without loss of generality, we may assume that $2 \leq p_1 < \ldots < p_l$.  First, suppose that $l \geq 4$.  By Inequality \eqref{eqn:multiplicitybounds}, we have 
\[
|B_1| \leq \frac{d + 1}{2}, \
|B_2| \leq \frac{d + 2}{3}, \
|B_3| \leq \frac{d + 4}{5}, \text{ and } 
|B_i| \leq \frac{d + 6}{7} \text{ for } i \geq 4.
\]
Therefore, we have 
\begin{align*}
\sum^l_{i=1} |B_i| & \leq \sum^l_{i=1} \frac{d + r_i - 1}{r_i} \\
& \leq \frac{d + 1}{2} + \frac{d + 2}{3} + \frac{d + 4}{5} + (l-3)\frac{(d + 6)}{7} \\
& = \frac{31d}{30} + \frac{59}{30} + (l-3)\frac{(d+6)}{7} \\
& = \frac{31d}{30} + \frac{59}{30} + (l-3)d + \frac{6(l-3)(1-d)}{7} \\
& = (l-2)d + 2 + \frac{d-1}{30} + \frac{6(l-3)(1-d)}{7} \\
& = (l-2)d + 2 - (d-1) \left [ \frac{6(l-3)}{7} -\frac{1}{30}\right ] \\
& < (l-2)d + 2,  
\end{align*}
since $l - 3 \geq 1$ and $d \geq 2$.  This contradicts Inequality \eqref{eqn:riemannhurwitz}.  

\vspace{0.2cm}

Therefore, we assume that $l \leq 3$.  Since $Y$ is an ISHS, we have $l=3$.  We begin with the case that $p_1 > 2$.  In this case, Inequality \eqref{eqn:multiplicitybounds} implies that for all $i$,  
\[
|B_i| \leq \frac{d + 2}{3}, 
\]
and this inequality fails to be strict for at most one $i$ (namely, equality is only possible if $i = 1$ and $p_1 = 3$, since we put the $p_i$ in increasing order).  Therefore, 
\[
\sum^l_{i=1} |B_i| < 3\frac{d+2}{3} = d + 2,   
\]
contradicting Inequality \eqref{eqn:riemannhurwitz}.  Thus, we let $p_1 = 2$.  Note that since $\pi_1(Y)$ is infinite by assumption, we cannot have $(p_1,p_2,p_3) = (2,3,5)$.  Therefore, we may assume $p_2 \geq 3$ and $p_3 \geq 7$.  We again apply Inequality \eqref{eqn:multiplicitybounds} to see
\begin{align*}
\sum |B_i| & \leq \frac{d+1}{2} + \frac{d+2}{3} + \frac{d+6}{7} \\
& = \frac{41d}{42} + \frac{85}{42} \\
& < d + 2,
\end{align*}    
contradicting Inequality \eqref{eqn:riemannhurwitz}.  
\end{proof}

\begin{proof}[Proof of Proposition~\ref{prop:classificationbranched}]
We let $g:Y' \to Y$ be a fiber-preserving branched cover between ISHS's, which necessarily has degree equal to the fiber degree by Proposition~\ref{prop:degreeisfiberdegree}.  Let $d$ denote this degree.  We suppose that $Y'$ (respectively $Y$) has unnormalized Seifert invariants $\{ \frac{q'_1}{q_1}, \ldots,\frac{q'_k}{q_k} \}$ (respectively $\{ \frac{p'_1}{p_1},\ldots,\frac{p'_l}{p_l}\}$), where we again allow $p_j = 1$.  We now apply \cite[Theorem 2.2]{Hu}, which describes how the Seifert invariants change under fiber-preserving branched covers, to see that $\{\frac{d q'_1}{q_1},\ldots,\frac{d q'_k}{q_k}\}$ also gives a set of unnormalized Seifert invariants for $Y$.  Let $d_i = \gcd(d,q_i)$.  Therefore, we may write $Y$ as $\Sigma(\frac{q_1}{d_1},\ldots,\frac{q_l}{d_l})$.  Next, \cite[Theorem 1.2]{NeumannRaymond} states that if $f:Y' \to Y$ is a fiber-preserving branched cover between closed, aspherical, oriented Seifert fibered spaces, then we have that $e(Y') = \frac{d}{d_{fib}^2}e(Y)$, where $d_{fib}$ denotes the fiber degree of $f$.  In our current setting, $d_{fib} = d$, which gives $e(Y') = \frac{e(Y)}{d}$.  For a Seifert homology sphere, we have $e(\Sigma(r_1,\ldots,r_m)) = -\frac{1}{r_1 \cdots r_m}$.  We therefore see that 
\[
d \cdot \prod \frac{q_i}{d_i} = q_1 \cdots q_l.  
\]
In particular, we have $d = d_1 \cdots d_l$.  In other words, $Y$ can be obtained from $Y'$ by a sequence of the moves described in the statement of the proposition and the degree of the map is as predicted.  
\end{proof}

\section{Discussion and further questions}
We now speculate about some possible generalizations of Theorem~\ref{thm:monotonicity} and Theorem~\ref{thm:nonzeroexist}. 

\subsection{Botany} Computational evidence suggests that Theorem~\ref{thm:monotonicity} should hold for $\HFhat$.  Therefore, we state this as a conjecture.  
\begin{conjecture}\label{conj:hatmonotonicity}  
Suppose that $(p_1,\ldots,p_l) \leq (q_1,\ldots,q_l)$.  Then, 
\[
\rank \HFhat(\Sigma(p_1,\ldots,p_l)) \leq \rank \HFhat(\Sigma(q_1,\ldots,q_l)).  
\] 
\end{conjecture}

Assuming Conjecture~\ref{conj:hatmonotonicity}, it is in fact easy to see that the above inequality can be made  strict if at least one $q_i$ is much larger than $p_i$ by pushing the argument in the proof of Proposition~\ref{prop:fiberbranchedhat}.  Therefore, Conjecture~\ref{conj:hatmonotonicity} would imply that there are only finitely many Seifert homology spheres with a fixed rank of $\HFhat$.  A similar argument as in the proof of Theorem~\ref{thm:botany} would then guarantee an algorithmic solution to the botany problem for the hat-flavor of the Heegaard Floer homology of Seifert homology spheres. 

\begin{conjecture}\label{conj:botany} 
Given $n\geq 1$, there are at most finitely many Seifert homology spheres $Y$ such that $\rank \HFhat (Y)=n$. Moreover, there is an algorithm to find all such $Y$. 
\end{conjecture}

\subsection{Non-zero degree maps} As pointed out in the introduction, we did not directly make use of the branched covers or vertical pinches in the proofs of Theorem \ref{thm:fiberbranched} and  Theorem~\ref{thm:pinch}.  It would be interesting to see the roles of these maps in the rank inequalities.

\begin{problem}
Prove Theorem~\ref{thm:fiberbranched} and  Theorem~\ref{thm:pinch} by directly using the corresponding maps.
\end{problem}

Proving Theorem~\ref{thm:fiberbranched} and  Theorem~\ref{thm:pinch} this way could also yield a generalization of Theorem~\ref{thm:nonzeroexist}. Now, we would like to ask about  possible generalizations of Theorem~\ref{thm:nonzeroexist}.  First, we point out that the statement is too strong to generalize to all three-manifolds.  One example can be seen by taking $Y = \Sigma_g \times S^1$, the product of a genus $g$ surface with a circle, for $g \geq 3$ (which is a Seifert fibered space).  Since $Y$ covers itself non-trivially by a degree $k$ map $f_k$, for any $k$, an analogue of Theorem~\ref{thm:nonzeroexist} would imply that $\rank \HFred(\Sigma_g \times S^1) = 0$.  However, this has been computed to be non-trivial \cite{hfk,jabukamark}.  We cannot have an analogous inequality even if we restrict to individual Spin$^c$ structures.  By choosing $\spinc \in \text{Spin}^c(\Sigma_g \times S^1)$ such that $\rank\HFred(\Sigma_g \times X^1,\spinc) \neq 0$, we see that $|k| \rank\HFred(\Sigma_g \times S^1,\spinc) \leq \rank\HFred(\Sigma_g \times S^1,f_k^*\spinc)$ cannot hold for $|k| \gg 0$, since for all $k$, $\rank\HFred(\Sigma_g \times S^1, f_k^*\spinc) \leq \rank \HFred(\Sigma_g \times S^1)$, the latter of which is finite.  It also seems unlikely that this is something special to having non-trivial first homology.  We expect that there are self-maps between integer homology spheres with non-trivial $\HFred$ which have $\mathrm{deg} \geq 2$ .  Since there is strong evidence that all aspherical integer homology spheres have non-trivial $\HFred$, again, the inequality in Theorem~\ref{thm:nonzeroexist} still seems very unlikely.   Therefore, we instead propose a weaker inequality.

\begin{conjecture}
If $f: Y' \to Y$ is a non-zero degree map between integer homology spheres, then 
\begin{align*}
&\rank \HFred(Y') \geq \rank \HFred(Y), \\
&\rank \HFhat(Y') \geq \rank \HFhat(Y).  
\end{align*}
\end{conjecture}

\vspace{0.2cm}     

\subsection{Correction terms} It is also natural to expect that there should be inequalities for the correction terms of Seifert homology spheres analogous to those in Theorems~\ref{thm:fiberbranched} and \ref{thm:monotonicity}.  However, it turns out that this is not the case.  The following calculation shows neither the $d$ invariant nor $-d$ satisfies the inequality in Theorems~\ref{thm:fiberbranched} and \ref{thm:monotonicity}.
$$d(-\Sigma (3,5,7))=-2> -4 = d (-\Sigma (3,5,14)) < -2 = d (-\Sigma(3,5,28)).$$

\noindent The correction term does not satisfy the inequality in Theorem~\ref{thm:pinch}  either, as the following observation shows.
$$d (-\Sigma (2,5,21))=0 > -2 = d(- \Sigma (2,3,5,7)) > -4 = d(-\Sigma(3,5,14)).$$
\noindent We therefore pose it as a question.
	
\begin{question}	What is the relationship between the correction terms in the settings of Theorem~\ref{thm:fiberbranched}, Theorem~\ref{thm:pinch}, and Theorem~\ref{thm:monotonicity} 
\end{question}

\bibliography{References}
\bibliographystyle{plain}

\end{document}